\documentclass[11pt]{amsart}
\usepackage{amsthm,amsfonts,amssymb,amsmath,oldgerm}
\numberwithin{equation}{section}
\usepackage{fullpage}
\usepackage{amsmath}
\usepackage{setspace}
\usepackage{stmaryrd}
\usepackage{mathrsfs}
\usepackage{fancyhdr}
\usepackage{graphicx}
\usepackage{enumerate}
\usepackage{bm}
\usepackage{bbm}
\usepackage{dsfont}
\usepackage{multirow}
\usepackage{psfrag}
\usepackage[font=scriptsize]{caption}
\usepackage[font=scriptsize]{subcaption}
\usepackage{listings}
\usepackage{tikz}
\usepackage{empheq}
\usepackage{cases}
\usetikzlibrary{matrix} 
\usepackage[framemethod=TikZ]{mdframed}
\mdfdefinestyle{MyFrame}{backgroundcolor=gray!20!white}
\usepackage[makeroom]{cancel}

\usepackage[top=30mm,bottom=30mm,left=22mm,right=22mm,a4paper]{geometry}

\usepackage{hyperref,bookmark}

\usepackage[normalem]{ulem}
\definecolor{violet}{rgb}{0.580,0.,0.827}

\usepackage[textwidth= \marginparwidth,textsize=tiny]{todonotes}

\hypersetup{
    colorlinks,          % color of internal links (change box color with linkbordercolor)
    citecolor=blue,        % color of links to bibliography
    filecolor=blue,      % color of file links
    urlcolor=blue,           % color of external links
    linkcolor=blue,
}
\newcommand{\ols}[1]{\mskip.5\thinmuskip\overline{\mskip-.5\thinmuskip {#1} \mskip-.5\thinmuskip}\mskip.5\thinmuskip} % overline short

\newcommand\dD{\mathrm{d}}

\def\eps{\varepsilon }

\newcommand{\sign}{{\text{\rm sgn }}}

\newcommand\br{\begin{remark}}
\newcommand\er{\end{remark}}
\newcommand\bp{\begin{pmatrix}}
\newcommand\ep{\end{pmatrix}}
\newcommand{\be}{\begin{equation}}
\newcommand{\ee}{\end{equation}}
\newcommand\ba{\begin{equation}\begin{aligned}}
\newcommand\ea{\end{aligned}\end{equation}}
\newcommand\ds{\displaystyle}

\newcommand{\beg}{\begin{example}}
\newcommand{\eeg}{\end{exaplem}}
\newcommand{\bpr}{\begin{proposition}}
\newcommand{\epr}{\end{proposition}}
\newcommand{\bt}{\begin{theorem}}
\newcommand{\et}{\end{theorem}}
\newcommand{\bc}{\begin{corollary}}
\newcommand{\ec}{\end{corollary}}
\newcommand{\bl}{\begin{lemma}}
\newcommand{\el}{\end{lemma}}
\newcommand{\bd}{\begin{definition}}
\newcommand{\ed}{\end{definition}}
\newcommand{\brs}{\begin{remarks}}
\newcommand{\ers}{\end{remarks}}

\newtheorem{theorem}{Theorem}[section]
\newtheorem{proposition}[theorem]{Proposition}
\newtheorem{corollary}[theorem]{Corollary}
\newtheorem{lemma}[theorem]{Lemma}
\newtheorem{remark}[theorem]{Remark}
\newtheorem{definition}[theorem]{Definition}

\newtheorem{example}[theorem]{Example}

\newcommand{\R}{{\mathbb R}}
\newcommand\bx{{\bm x}}

\newcommand{\by}{\bm y}

\newcommand{\bu}{\bm u}

\newcommand\cA{{\mathcal A}}
\newcommand\cB{{\mathcal B}}

\newcommand\cE{{\mathcal E}}

\newcommand\cU{{\mathcal U}}
\newcommand\cV{{\mathcal V}}
\newcommand\cW{{\mathcal W}}

\newcommand\scA{{\mathscr A}}
\newcommand\scC{{\mathscr C}}

\newcommand\scH{{\mathscr H}}

\newcommand\scJ{{\mathscr J}}

\usetikzlibrary{fit}
\usetikzlibrary{arrows}
\numberwithin{equation}{section}
\hypersetup{linkbordercolor=red}

\title{
Concentration profiles in FitzHugh-Nagumo neural networks:
A Hopf-Cole approach}
\date {\today}

\begin{document}
\maketitle

\begin{center}
\author{\textsc{Alain Blaustein} 
\footnote{Institut de Mathématiques de Toulouse, Université Paul Sabatier, 118 route de Narbonne - F-31062 Toulouse Cedex 9, France. E-mail :  \texttt{alain.blaustein@math.univ-toulouse.fr}}
\and \textsc{Emeric Bouin} \footnote{CEREMADE - Universit\'e Paris-Dauphine, UMR CNRS 7534, Place du Mar\'echal de Lattre de Tassigny, 75775 Paris Cedex 16, France. E-mail: \texttt{bouin@ceremade.dauphine.fr}}}

\end{center}
\bigskip

\begin{abstract}
In this paper we focus on a spatially extended FitzHugh-Nagumo model with interactions. In the  regime where strong and local interactions dominate, we quantify how the probability density of neurons concentrates into a Dirac distribution. Previous work investigating this question have provided relative bounds in integrability spaces. Using a Hopf-Cole framework, we derive precise $L^\infty$ estimates using subtle explicit sub- and super- solutions which prove, with rates of convergence, that the blow up profile is Gaussian. 
\end{abstract}

\maketitle
\vspace{0.5cm}
\noindent
\textbf{\textit{Keywords: }} .
\\

%\vspace{0.25cm}
\noindent
\textbf{\textit{Mathematics Subject Classification (2010): }}
%\tableofcontents

%\newpage

\section{Introduction}

\subsubsection*{\bf The model.} Neurons inside the brain are known to interact with each other through the variation of their membrane potential. One of the main challenges in neuroscience consists in understanding complex dynamics induced by these interactions. A quantitative description started with the pioneering work \cite{HH} of A. Hodgkin and A. Huxley, where they built an accurate model to describe the membrane potential dynamics of a single nerve cell submitted to an external current. Simplified versions of this model exist, for example FitzHugh-Nagumo models \cite{FHN1,FHN2}: these ones  keep the main features of Hodgkin-Huxley models but are more tractable from a mathematical point of view. Based on FitzHugh-Nagumo equations, we focus on the following model
    \begin{equation}
  \label{kinetic:eq}
  \ds\partial_t   f^\eps
         + 
        \mathrm{div}_{\bu}
        \left[  
        \mathbf{b}^\eps
         f^\eps 
         \right]
        - 
       \partial^2_v  
        f^\eps
         = 
        \frac{1}{\eps} 
        \rho^\eps_0 
        \partial_v 
        \left[ 
        (v-\cV^\eps)
         f^\eps 
        \right] 
        .
\end{equation}

Equation \eqref{kinetic:eq} is obtained as the mean-field limit of a microscopic model of the following form (the precise expression of $N$ and $A$ will be given later on), which is a coupled system of stochastic differential equations describing a neural network of size $n$,
\begin{equation*}
\left\{
    \begin{array}{lll}
      & \displaystyle \dD v^i_t = \left(N(v^i_t) - w^i_t  - \frac{1}{n}\sum_{j=1}^n \Phi_\eps(\bx_i,\bx_j) (v^i_t-v^j_t)  \right)\dD t   +  \sqrt{2} \dD B^i_t ,
      \\[2em]
        & \displaystyle \dD w^i_t = A\left(v^i_t,w^i_t\right) \dD t ,
    \end{array}
\right.
\end{equation*}

In the limit $n\rightarrow + \infty$ and with a suitable choice for the interaction kernel $\Phi_\eps$, the empirical measure associated to the latter system converges to $f^\eps$: see for instance \cite{mlimit1, mlimit2, mlimit3, mlimit4} for a rigorous proof of this limit in the case a the FitzHugh-Nagumo system and \cite{mlimit_bolley} for a related model in collective dynamics.

In \eqref{kinetic:eq}, $\bu = (v,w)\in\R^2$, where the first variable $v$ stands for the membrane potential of a neuron and is coupled to an auxiliary variable $w$ called the adaptation variable. The distribution function $f^\eps := f^\eps(t,\bx,\bu)$ represents the density of neurons at time $t$, located at position $\bx$ in a compact set $K\subseteq\R^d$, with a membrane potential $v\in\R$ and an adaptation variable $w\in \R$.
In the latter equation, the right hand side accounts for local interactions between neurons and it is re-scaled by a small parameter $\eps>0$ which describes their strength. It also displays the macroscopic quantities associated to the network: the spatial distribution of neurons throughout the network
\begin{equation*}
\rho_0^\eps(\bx)
=
\int_{\R^2}f^\eps(t,\bx,\bu) \dD \bu ,
\end{equation*}
and the averaged  voltage and adaptation variable at a spatial location $\bx$
\begin{equation}\label{macro:q}
\left\{
    \begin{array}{ll}
        \displaystyle \rho_0^\eps
        \left(
        \bx
        \right)\cV^\eps(t,\bx) 
        & = 
        \ds\int_{\R^2}v~f^\eps(t,\bx,\bu) \dD\bu ,
        \\[1.1em]
        \displaystyle \rho_0^\eps
        \left(
        \bx
        \right)\cW^\eps(t,\bx) 
        & = 
        \ds\int_{\R^2}w~f^\eps(t,\bx,\bu) \dD\bu .
    \end{array}
\right.
\end{equation}
In the sequel, we use the vector notation 
$\ds \mathcal{U}^\eps 
= 
\left(
\mathcal{V}^\eps ,
\mathcal{W}^\eps
\right)
$. Notice that $\rho_0^\eps$ is indeed time-homogeneous, integrating the mean field equation \eqref{kinetic:eq} with respect to $\bu\in\R^2$. In \eqref{kinetic:eq}, coefficient $\mathbf{b}^\eps$ is defined for all 
$
\ds
\left(
t,\bx,\bu
\right)
\in
\R^+
\times
K
\times 
\R^2
$
as
\begin{equation}\label{def:Beps}
\mathbf{b}^\eps(t,\bx,\bu)
 := 
\begin{pmatrix}
\ds 
B^\eps(t,\bx,\bu) \\[0,9em]
\ds
A(\bu)
\end{pmatrix}
= 
\begin{pmatrix}
N(v) - w - \mathcal{K}_{\Psi}[ f^\eps ]\left(t,\bx,v\right) . \\[0,9em]
\ds
 a v  -  b w  +  c
\end{pmatrix}
 ,
\end{equation}
where $a$, $c \in \R$ and $b>0$. The drift $N \in \scC^2(\R)$ satisfies
\begin{equation}\label{hyp:N}
\limsup_{|v| \to + \infty} \frac{N(v)}{\sign{v}|v|^p} < 0, \qquad \sup_{|v| \geq 1} \left\vert \frac{N(v)}{|v|^p} \right\vert < + \infty .
\end{equation}
for some $p \geq 2$, and
\begin{equation}\label{hyp N'}
    \sup_{|v| \geq 1} 
    \left( 
    \left|N''(v)\right|
     + 
    \left|N'(v)\right| 
    \right) 
    | v |^{-p'}
     <  + \infty ,
\end{equation}
for some $p'$.
A historical choice for $N$ is a confining non-linearity with the form
\[
N(v) =  v  -  v^3,
\] 
but many other cases are possible. 

The operator $\mathcal{K}_{\Psi}$, which accounts for the long-range interactions, is defined as
\[
\mathcal{K}_{\Psi}[ f ](t,\bx,v)
 = 
\int_{K\times\R^{2}}
\Psi(\bx,\bx')  (v-v') f(t,\bx',\bu') 
\dD\bx' \dD \bu' .
\]
The connectivity kernel
$\Psi\in\scC^0\left(K_{\bx}, L^1\left(K_{\bx'}\right)\right)$ satisfies
\begin{equation}
\label{hyp:psi}
\sup_{\bx'\in K} \int_{K}\left|\Psi(\bx,\bx')\right|  \dD\bx  <\; +\infty , \qquad \sup_{\bx\in K} \int_{K}\left|\Psi(\bx,\bx')\right|^r  \dD\bx'  <  +\infty ,
\end{equation}
for some $r>1$. Thanks to this set of assumptions on $\Psi$ our model takes into account non-symmetric interactions between neurons and authorize $\Psi$ to behave like a power law, a case which is considered in the physical literature (see \cite{mlimit4}). We point out that $\mathcal{K}_{\Psi}$ can be expressed in terms of the macroscopic quantities 
\[
\mathcal{K}_{\Psi}
        [ f^\eps ](t,\bx,v)
 = 
\Psi*_r\rho^\eps_0(\bx)  v  -  \Psi*_r(\rho^\eps_0 \;\cV^\eps)(t,\bx) ,
\]
where $*_r$ is a shorthand for the convolution on the right side of any function $g$ with $\Psi$
\[
\Psi*_r g(\bx)
\;= 
\int_{K}
\Psi(\bx,\bx') g(\bx') \dD \bx' .
\]

Before going further, we shall be precise about the notion of solution we consider for equation \eqref{kinetic:eq}. For this, we suppose, for each $\eps > 0$,
\begin{equation}\label{hyp0:mu0}
f^\eps_0
\in 
\scC^0
\left(K,L^1\left(\R^2
\right)\right)
 ,\quad
f^\eps \geq 0
\quad
\textrm{and}
\quad
\int_{\R^2} f^\eps_0(\bx,\bu) \dD \bu \dD \bx = 1,\quad \forall \bx \in K .
\end{equation}
Therefore, $\rho_0^\eps \in 
\scC^0
\left(K\right)$. We also suppose that 
\begin{equation}\label{hyp:rho0}
m_* \leq \rho_0^\varepsilon \leq 1/m_* ,
\end{equation}
for all $\eps > 0$ and for some positive constant $m_*$ independent of $\eps$. 
\begin{remark}
The uniform lower bound condition on $\rho_0^\eps$ seems necessary in our analysis because the problem degenerates when $\rho_0^\eps$ vanishes as it may be seen on equation \eqref{kinetic:eq}. However the upper bound condition on $\rho_0^\eps$ may be relaxed at the cost of loosing uniform convergence with respect to $\bx$ in our main result Theorem \ref{th:main}.
\end{remark}
On top of that, we assume the following condition: there exists two positive constants $m_p$ and $\ols{m}_p$, independent of $\eps$, such that
\begin{equation}
\label{hyp1:f0}
\sup_{\bx \in K}
 \int_{\R^2}
|\bu|^{2(p+p')} f^\eps_0
\left(\bx,\bu
\right) 
\dD \bu
 \leq  m_p ,
\end{equation}
and such that
\begin{equation}
\label{hyp2:f0}
\int_{K \times   \R^2} 
|\bu|^{2(p+p') r'}
 f^\eps_0
\left(\bx,\bu
\right)
\dD \bu \dD \bx
 \leq  \ols{m}_p ,
\end{equation}
where $p$, $p'$ and $r'$ are given in \eqref{hyp:N}, \eqref{hyp N'} and \eqref{hyp:psi}.

\begin{definition}\label{notion de solution}
For all $\varepsilon>0$, we say that $f^\varepsilon$ solves \eqref{kinetic:eq} with initial condition $f^\varepsilon_0$ if $f^\eps \in \scC^0
\left( 
\R^+\times K , 
L^1
\left(\R^2\right) 
\right)$ and for all $\bx \in K$, $t \geq 0$,
and $\ds\varphi \in \scC_c^\infty
\left(\R^2\right)$
, it holds
\begin{align*}
\int_{\R^2}
\varphi(\bu) 
\left(
f^\varepsilon
\left(t,\bx,\bu\right)
-
f^\varepsilon_{0}
\left(\bx,\bu\right)
\right) \dD\bu 
&= 
\int_0^t
\int_{\R^2}
\left[ 
\left(
\nabla_{\bu} \varphi
\cdot 
\mathbf{b}^\eps
         + 
\partial_v^2
 \varphi
\right)
f^\varepsilon 
\right](s,\bx,\bu) \dD\bu \dD s\\[0,8em]
&-\frac{\rho_0^\eps(\bx)}{\eps}
\int_0^t
\int_{\R^2}
\left[ 
\partial_v
 \varphi 
        \left(v-\cV^\eps\right)
f^\varepsilon 
\right](s,\bx,\bu) \dD\bu \dD s ,
\end{align*}
where $\mathcal{V}^\eps$ and $\mathbf{b}^\eps$  are given by \eqref{macro:q} and \eqref{kinetic:eq} respectively.
\end{definition}
With this notion of solution, equation \eqref{kinetic:eq} is well-posed, the following result being proved in \cite{BF}. 
\begin{theorem}[\cite{BF}]
\label{WP mean field eq}
For any $\varepsilon > 0$, suppose that assumptions \eqref{hyp:N} on $N$, \eqref{hyp:psi} on $\Psi$ and \eqref{hyp0:mu0}-\eqref{hyp:rho0} on the initial condition are fulfilled and that $f_0^\varepsilon$ also verifies
\begin{equation*}
\left\{
\begin{array}{l}
\ds\sup_{\bx \in K}
\int_{\R^2}
e^{|\bu|^2/2} 
f^\eps_{0}(\bx,\bu) 
\dD\bu 
<  +\infty ,
\\[1.5em]
\ds \sup_{\bx \in K}
\int_{\R^2}
\ln{
\left[ 
f^\eps_{0}(\bx,\bu) 
\right]} f^\eps_{0}(\bx,\bu)
\dD\bu 
<  +\infty ,
\end{array}\right.
\end{equation*}
and
\begin{equation*}
\sup_{\bx \in K}
    \left\|
\nabla_{\bu}
\sqrt{f^\eps_{0}} 
\right\|^2_{L^2(\R^2)} 
    <  +\infty .
\end{equation*}
Then there exists a unique solution $f^\varepsilon$ to equation \eqref{kinetic:eq} with initial condition $f^\varepsilon_0$, in the sense of Definition \ref{notion de solution} which verifies
\[
\displaystyle
\sup_{(t , \bx) \in [0,T]\times K} 
\int_{\R^2}
e^{|\bu|^2/2}  f^\varepsilon(t,\bx,\bu) 
\dD\bu  < +\infty ,
\]
for all $ T \geq 0$,
\end{theorem}

\subsubsection*{\bf The question at hand.} The purpose of this article is to go through the mathematical analysis of the neural network in the regime of strong local interactions, that is when $\eps \ll 1$. More precisely, we prove that the voltage distribution concentrates into a Dirac mass and propose a quantitative description of the concentration profile. 
First, we highlight this concentration phenomenon with some formal computations. We look for the leading order in \eqref{kinetic:eq}: in our case, it is induced by short range interactions between neurons, and as $\eps\rightarrow 0$, we expect
$$
( v  -\; \cV^\eps) 
f^\eps \underset{\eps \rightarrow 0}{\to}  0 ,
$$
to make sure that no terms are singular. This means that $f^\eps$ concentrates towards a Dirac mass centred in $\cV^\eps$ with respect to the $v$-variable, that is
\[
f^\eps(t,\bx,\bu)
 \underset{\eps \rightarrow 0}{\approx}
\delta_0
\left(
v-
\cV^\eps(t,\bx)
\right) \otimes  F^\eps
(t,\bx,w)
 ,
\]
where $\ds \cV^\eps$ is given by \eqref{macro:q} and $F^\eps$ is defined as the marginal of $f^\eps$ with respect to the voltage variable
\begin{equation*}
F^\eps
\left(t,\bx,w
\right)
 = 
\int_{\R} f^\eps
\left(t,\bx,\bu
\right) \dD v .
\end{equation*}

Multiplying equation \eqref{kinetic:eq} by $v$ and then integrating over $\bu\in\R^2$~(resp.~$v\in\R$), multiplying the second line of \eqref{macro-eps:eq} by $w/\rho_0^\eps$ and integrating with respect to $w$, one finds that the couple $\ds(\cV^\eps,\cW^\eps)$ solves the following system
\begin{equation}\label{macro-eps:eq}
    \left\{
    \begin{array}{llll}
        &\displaystyle \partial_t \cV^\eps  =  
        N(\cV^\eps)
         - 
        \cW^\eps
         - 
       \mathcal{L}_{\rho^\eps_0}
       \left[ \mathcal{V^\eps} \right]
         + 
        \cE(f^\eps)
         ,\\[0.8em]
        &\partial_t  \cW^\eps
         = 
        A(\cV^\eps,\cW^\eps) ,
  \end{array}
\right.
\end{equation}
where the error term $
\ds\cE
\left(f^\eps
\right)$ is given by
\begin{equation}\label{error}
\cE
\left(
f^\eps\left(t,\bx,\cdot 
\right)
\right)
\;= 
\;
\frac{1}{\rho_0^\eps(\bx)} 
\int_{\R^2} N(v) 
f^\eps
\left(t,\bx,\bu
\right) 
\dD \bu -  N(\cV^\eps) ,
\end{equation}
and 
$\ds\mathcal{L}_{\rho_0^\eps}
$ is a non local operator given by
\[
\mathcal{L}_{\rho_0^\eps}
\left[ \cV^\eps 
\right]
=
\cV^\eps  \Psi*_r\rho_0^\eps  -  \Psi*_r(\rho_0^\eps \cV^\eps)\;.
\]

All this, in turn, implies that as $\eps$ vanishes,
$\ds\left(\cV^\eps,\cW^\eps
\right)$
converges to the couple $\ds\left(\cV,\cW
\right) $, which solves
\begin{equation}\label{macro:eq}
    \left\{
    \begin{array}{l}
        \displaystyle \partial_t \cV
         =  
        N(\cV)
         - 
        \cW
         - 
       \mathcal{L}_{\rho_0}[ \cV ]
         ,\\[0.9em]
        \displaystyle \partial_t  \cW
         = 
        A(\cV,\cW) ,\\[0.9em]
        \displaystyle 
        \left(
        \cV\left(0,\cdot\right)
        ,
        \cW\left(0,\cdot\right)
        \right)
         = 
        \left(
        \cV_0 
        ,
        \cW_0 
        \right)
         .
    \end{array}
\right.
\end{equation}
As of the marginal $F^\eps$, it has been shown in \cite{CFF,Blaustein,BF} that it also converges.

In this article, we refine the latter result by investigating the concentration profile of the solution $f^\eps$ when $\eps$ goes to $0$. For that purpose, we perform the so-called Hopf-Cole transform of $f^\eps$, 
\begin{equation}\label{def phi eps}
\phi^\eps := \eps \ln \left( \sqrt{\frac{2\pi\eps}{\rho_0}} f^\eps \right),
\end{equation}
and study the convergence of $\phi^\eps$ as $\eps$ goes to zero. This approach has been widely used in selection-mutation models in population dynamics (for example), to study concentration phenomena occurring in the small mutation regime for populations structured by space and trait: see for instance \cite{Barles/Mirrahimi/Perthame,Mirrahimi/Roquejoffre}.

\subsubsection*{\bf Heuristics and main result.} Let us now formally present the convergence of $\phi^\eps$:
injecting \textit{ansatz} \eqref{def phi eps} in equation \eqref{kinetic:eq}, we find that $\phi^\eps$ solves the following Hamilton-Jacobi equation
\begin{equation}\label{HJ eq ordre 0}
\partial_t \phi^\eps + \nabla_{\bu}\phi^\eps \cdot
\mathbf{b}^\eps + \eps  \mathrm{div}_{\bu} \left[  \mathbf{b}^\eps \right] - \partial^2_v \phi^\eps - \rho_0^\eps =  \frac{1}{\eps} \left(  \partial_v \left(
\frac{1}{2} 
\rho_0^\eps \left| v-\cV^\eps \right|^2 + \phi^\eps  \right) \partial_v \phi^\eps \right)
\end{equation}

Keeping only the leading order in equation \eqref{HJ eq ordre 0}, it yields
\[
\phi^\eps  \underset{\eps \rightarrow 0}{\approx} 
-\frac{\delta_1}{2} \,
\rho_0^\eps \left| v-\cV^\eps \right|^2\mathds{1}_{v<\cV^\eps}
-\frac{\delta_2}{2} 
\rho_0^\eps \left| v-\cV^\eps \right|^2\mathds{1}_{v>\cV^\eps}
+c
\,,
\]
for some positive constant $c$ and where $\delta_i$ lies in $\{0,1\}$.
Provided that the macroscopic quantities $\rho^\eps$ and $\ds
\cV^\eps
$ converge as $\eps$ goes to $0$, we expect
\[
\phi^\eps(\cdot, \cdot, v) 
\underset{\eps\rightarrow 0}{\longrightarrow}\; -\,\frac{\delta_1}{2} \,
\rho_0 \left| v-\cV \right|^2\mathds{1}_{v<\cV}
-\frac{\delta_2}{2} 
\rho_0 \left| v-\cV \right|^2\mathds{1}_{v>\cV}
+c
\,.
\]
Furthermore, since our problem conserves mass, we expect
\[
\int_{v\in\R}
\exp{
\left(
-\,\frac{\delta_1}{2\,\eps} \,
\rho_0 \left| v-\cV \right|^2\mathds{1}_{v<\cV}
-\frac{\delta_2}{2\,\eps} 
\rho_0 \left| v-\cV \right|^2\mathds{1}_{v>\cV}
+\frac{c}{\eps}
\right)
}
\,\dD v\,=\,\sqrt{\frac{\rho_0}{2\pi\eps}}\,,
\]
for all $\eps$. This forces $\delta_1=\delta_2=1$ and $c=0$ and therefore we obtain
\[
\phi^\eps(\cdot, \cdot, v) 
\underset{\eps\rightarrow 0}{\longrightarrow}\; -\,\frac{1}{2} \,
\rho_0 \left| v-\cV \right|^2
\,.
\]
This convergence is the object of our main result, Theorem \ref{th:main} below. It will justify the latter limit and provide explicit convergence rates. Our strategy consists in performing a Hilbert expansion of $\phi^\eps$ with respect to $\eps$ and to prove that the higher order terms in the expansion are uniformly bounded with respect to $\eps$ using a comparison principle (see Section \ref{A priori estimates} for more details).

\begin{theorem}\label{th:main}
Assume \eqref{hyp:N}-\eqref{hyp2:f0} and the additional assumptions of Proposition \ref{classical solution} and Theorem \ref{WP mean field eq}. Suppose that there exists a positive constant $C$ independent of $\eps$ such that the following compatibility assumption holds
\begin{equation}\label{compatibility assumption}
\left\| 
\mathcal{U}_0
 - 
\mathcal{U}_0^\eps 
\right\|_{L^{\infty}(K)}
 + 
\| 
\rho_0
-
\rho_0^\eps 
\|_{L^{\infty}(K)} 
\leq 
C \eps ,
\end{equation}
as well as the following set of "smallness assumptions"
\begin{subequations}
\begin{numcases}{}
\label{hyp:HJ}
\ds\left|\phi^\eps_0
+ \frac{1}{2} 
\rho_0 \left| v-\cV \right|^2 
- \eps n
\right|
 \leq 
\eps C
\left(
1 + 
|\bu|^2
\right) ,\quad
\forall 
(\bx,\bu)\in K \times \R^2 ,\\[0,9em]
\label{hyp:f eps 0 well prepared}
\int_{\R^2} 
\left(
|v- \cV^\eps_0|^{2} + |v- \cV^\eps_0|^{p'+1}
\right) 
f^\eps_0(\cdot,\bu) \dD \bu
 \leq  C \eps .
\end{numcases}
\end{subequations}

Then the sequence $\left(\phi^\eps\right)_{\eps>0}$ of Hopf-Cole transforms of $
\left(f^\eps\right)_{\eps > 0}$ is well defined and it converges locally uniformly on $
\R^+\times K\times\R^2
$ to $-
\frac{1}{2} 
\rho_0 \left| v-\cV \right|^2$ with rate $\eps$. More precisely, there exist two positive constants $C$ and $\eps_0$ such that for all $\eps \leq \eps_0$,
\[
\left|\phi^\eps
+ 
\frac{1}{2} 
\rho_0 \left| v-\cV \right|^2 
- \eps n
\right|(t,\bx,\bu)
 \leq 
\eps 
Ce^{Ct}
\left(
1 + 
|\bu|^2
\right) ,\quad
\forall 
(t,\bx,\bu)\in \R^+\times  K \times \R^2 .
\] 
As a consequence, $f^\eps$ converges uniformly to $0$ on the compact subsets of 
$ \R^+\times K \times \R^2
\setminus
\left\{
v\neq 
\cV\left(t,\bx\right)
\right\}
$.
In the latter results, constants $C$ and $\eps_0$ only depends on the data of our problem: $f^\eps_0$ (only through the constants appearing in assumptions
\eqref{hyp:rho0}-\eqref{hyp2:f0} and
\eqref{compatibility assumption}-\eqref{hyp:f eps 0 well prepared}), $N$, $A$ and $\Psi$.
\end{theorem}

Before going further into our analysis, let us comment on our result. We first emphasize that our result deals with uniform convergence with respect to all variables, which is a great improvement in comparison to former results obtained in \cite{Blaustein,BF}, where $L^1$, $L^2$ and weak convergence estimates were obtained. We also point out that the present article is in line with \cite{HJ quininao/touboul}, which also goes through the analysis of the convergence of $\phi^\eps$. However, the latter article relies on a compactness argument and this has two major consequences: first the limit $-\frac{1}{2}  \rho_0 \left| v-\cV \right|^2$ is not identified and second no rate of convergence is obtained. To end with, authors in \cite{HJ quininao/touboul} deal with a spatially homogeneous neural network. 

Our result does not yet describe the limiting pointwise dynamics of $f^\eps$ with respect to the adaptation variable $w$: this corresponds to limit of $f^\eps$ on the set of points $\ds(t,\bx,\bu)
\in \R^+ \times K \times \R^2
$ such that $v=\cV(t,\bx)$. Indeed, due to \textit{ansatz} \eqref{def phi eps}, we only treat the limiting dynamics with respect to the fast variable $v$. This difficulty is quite common in problems which display fast and slow variables: see \cite{Bouin/Mirrahimi} for an another instance in a model for a population structured by trait and spatial location. In order to fill the gap, one would need to investigate the convergence of $\phi^\eps_1$ defined as
\begin{equation}
   \phi^\eps_1 := \frac{1}{\eps} \left(  \phi^\eps +
\frac{1}{2} 
\rho_0 \left| v-\cV \right|^2  \right).
\end{equation}
which corresponds to the next term in the expansion of $\phi^\eps$ with respect to $\eps$. We provide a partial answer to this concern in Section \ref{A priori estimates} by identifying a formal equivalent of $\phi^\eps_1$ but the rigorous analysis of the convergence of $\phi^\eps_1$ with locally uniform bounds is beyond the scope of this article.

The last comment on our result is that it holds in a perturbative setting. Indeed, our assumption at time $t=0$ translates on $f^\eps_0$ as follows
\begin{equation*}
    f^\eps_0
    \left(\bx,\bu\right) 
     \underset{\eps\rightarrow 0}{=} 
    \sqrt{\frac{\rho_0(\bx)}{2\pi\eps}}
    \exp{
    \left(
    - 
\frac{\rho_0(\bx)}{2 \eps} 
        \left|
        v-\cV_0
        \left(\bx\right)
        \right|^2
         + 
        O(1)
    \right)
    } ,
\end{equation*}
which means that the $f^\eps$ is already concentrated at initial time. This restriction is quite common in articles which follow a Hopf-Cole transform approach: it is the case in all the references cited above \cite{Barles/Mirrahimi/Perthame,Mirrahimi/Roquejoffre,HJ quininao/touboul,Bouin/Mirrahimi}. Let us point out that \cite{HJ quininao/touboul} deals with initial conditions concentrated with respect to the adaptation variable $w$ as well, a condition which is lifted in our result. In \cite{BF,Blaustein}, $f^\eps$ is not initially concentrated and we deal with general initial conditions, however this is possible since convergences are in $L^1$ and $L^2$.

\subsubsection*{\bf Comments on the stategy.} Let us outline our strategy and the challenges in order to prove Theorem \ref{th:main}. The main difficulty is induced by the drift $N$, which is not Lipschitz according to assumption \eqref{hyp:N}. Indeed, a common technique consists in obtaining some uniform bounds with respect to $\eps$ on the derivatives of $\phi^\eps$ by applying the Bernstein method (see \cite{Barles} for a general setting in which the Bernstein method applies and \cite{Barles/Mirrahimi/Perthame, HJ quininao/touboul} for some application in particular contexts) and then to conclude on the convergence of $\phi^\eps$ with a compactness argument. However, this method, does not seem to apply easily here since $N$ is not uniformly Lischitz and therefore induces high-order terms with respect to $v$ in the equations on the derivatives of $\phi^\eps$. An alternate mean to carry out the proof would be to use the method of half-relaxed limits introduced by Barles and Perthame in \cite{Barles/Perthame} which applies without requiring any regularity estimates, at the cost of loosing continuity and therefore uniqueness in the limit $\eps \rightarrow 0$. To recover uniqueness, we add the additional constraint 
\[
\phi(t,\bx,\cV,w)\,=\,0\,
\]
on the limiting problem. But proving that the limit provided by the method of half-relaxed limits satisfies this constraint brings us back to our initial unsolved problem since it requires regularity estimates on the derivatives of $\phi^\eps$. To bypass these difficulties, we choose another approach which does not require regularity estimates and which has the advantage of providing explicit convergence rates: instead of proving uniform estimates on the derivatives of $\phi^\eps$, we prove uniform estimates on the first term in the expansion of $\phi^\eps$ with respect to $\eps$. This is made possible since this first term takes into account the non-linear fluctuations induced by $N$. Indeed, these non-linear variations induced by $N$ are expected to be perturbations of order $\eps$, as it may be seen rewriting equation \eqref{HJ eq ordre 0} on $\phi^\eps$ as follows
\[
\partial_v
        \left( 
        -\frac{1}{2} 
\rho_0^\eps \left| v-\cV^\eps \right|^2
        +\eps n(v)
        -\phi^\eps 
        \right) \partial_v  \phi^\eps
+
\hdots
= 0 ,
\]
where the correction $n(v)$ is such that $n'(v)=N(v)$ and
where "$\hdots$" gathers the lower order terms with respect to $v$ and $w$.
Hence, at least formally as $\eps$ goes to zero, we expect
\[
\phi^\eps   \underset{\eps \rightarrow 0}{\approx} -\frac{1}{2} 
\rho_0^\eps \left| v-\cV^\eps \right|^2
         + \eps n(v)
 .
\]
Therefore, our strategy
consists in considering the first term $\phi^\eps_1$ in the expansion of $\phi^\eps$ with respect to $\eps$, that is
\[
\phi^\eps  = 
-\frac{1}{2} 
\rho_0^\eps \left| v-\cV^\eps \right|^2
 + \eps 
\phi^\eps_1 ,
\] 
and to prove that it looks like $n(v)$ plus some uniformly bounded with respect to $\eps$ lower order terms which are induced by the globally Lipschitz coefficient in equation \eqref{HJ eq ordre 0}. To do so, we identify a formal equivalent $\overline{\phi^\eps_1}$
of $\phi^\eps_1$ as $\eps$ goes to zero, which displays $n(v)$ and which depends on $\eps$ only through the macroscopic quantities 
$\left(\cV^\eps,\cW^\eps\right)$ (see Section \ref{A priori estimates}). In Lemma \ref{lemme technique s s solution}, we look for super and sub-solutions to the equation solved by $\phi^\eps_1$ with the form $\chi_{\pm} = \overline{\phi^\eps_1} \pm  \psi$. Once this is done, we apply a comparison principle in order to obtain $\chi_- \leq \phi^\eps_1 \leq \chi_+$. The last step consists in proving that the equivalents $-\frac{1}{2} 
\rho_0^\eps \left| v-\cV^\eps \right|^2$ and $\overline{\phi^\eps_1}$, which depend on $\eps$ only through the macroscopic quantities
$\left(\cV^\eps,\cW^\eps\right)$ have convergence and boundedness properties. This is done relying on previous results which ensure that the macroscopic quantities converge (see Theorem \ref{th:preliminary}).

\subsubsection*{\bf Comments and perspectives.}
Two major perspectives arise from our work. First, our result holds in a perturbative setting in the sense that we need the initial data to be concentrated in order for our result to hold true. It would be interesting to lift this constraint and to treat a general set of initial data without requiring any well-preparedness condition. To achieve this, one possibility would be to adapt the strategy adopted in \cite{Blaustein}, where we introduced a time dependant scaling in order to take into account the initial layer induced by the ill-preparedness of the initial data. Second, our result does not describe the limiting dynamics with respect to the adaptation variable, which is a "slow variable" in our problem. Indeed, due to our starting point, \textit{ansatz} \eqref{def phi eps}, our result is doomed to only describe the limiting dynamics with respect to $v$, the fast variable of our problem. It would be a great but challenging improvement to carry out the convergence analysis of the first corrective term in the Hilbert expansion of $\phi^\eps$ with respect to the scaling parameter $\eps$ in order to fill this gap.

\subsubsection*{\bf Structure of the paper.} The remaining part of this article is organized as follows: in Section \ref{sec:preli}, we prove some regularity estimates for equation \eqref{kinetic:eq} in order to make our further computations rigorous: this is the object of Lemma \ref{WP2} and Proposition \ref{classical solution}. In Theorem \ref{th:preliminary} and Proposition \ref{estimate:derivee:erreur}, we also we recall and prove some convergence results on the macroscopic quantities $\cU^\eps$ and $\cE\left(f^\eps\right)$. Then we pass to Section \ref{A priori estimates}, which is dedicated to the proof of Theorem \ref{th:main}. The proof relies on the key Lemma \ref{lemme technique s s solution}, in which we construct sub- and super-solution for equation \eqref{HJ eq ordre 0} on $\phi^\eps$.

\section{Preliminary estimates}\label{sec:preli}

\begin{lemma}\label{WP2}
For all function $\varphi\in\scC^2
\left(
\R^2
\right)
$ with polynomial growth of order $r \geq 0$, that is 
\[
\left|\varphi(\bu)\right|
 + 
\left|\nabla_{\bu}\varphi(\bu)\right|
 + 
\left|\nabla^2_{\bu}\varphi(\bu)\right|
 \underset{|\bu|\rightarrow+\infty}{=} 
O\left(|\bu|^r
\right) ,
\]
the function 
$\ds
\left(
\left(
t,\bx
\right)
\mapsto
\int_{\R^2}\varphi(\bu) f^\eps
\left(t,\bx,\bu\right) \dD\bu
\right)
$ is continuous and has continuous time derivative over $\R^+ \times K$. In particular, the macroscopic quantities 
$
\ds
\mathcal{V}^\eps
$ and 
$
\ds
\mathcal{W}^\eps
$ given by \eqref{macro:q} and the error 
$
\ds
\cE\left(f^\eps\right) 
$ given by \eqref{error} are continuous and have continuous time derivatives.

\end{lemma}

\begin{proof}
Take such a function $\ds\varphi$. To simplify notations we write
\[
\varphi
\left(
f^\eps
\right):
\left(
t,\bx
\right)
\mapsto
\int_{\R^2}\varphi(\bu) f^\eps
\left(t,\bx,\bu\right) \dD\bu .
\]
We start by proving that $\ds\varphi\left(f^\eps\right)$ is continuous. This is straightforward according to the following estimate obtained after applying Cauchy-Schwarz inequality
\[
\left|
\varphi\left(
f^\eps
\right)(t,\bx)
-
\varphi\left(
f^\eps
\right)(s,\by)
\right|
 \leq 
\left(
\varphi^2\left(
f^\eps
\right)(t,\bx)
+
\varphi^2\left(
f^\eps
\right)(s,\by)
\right)^{1/2}
\|
f^\eps
\left(t,\bx\right)
-
f^\eps
\left(s,\by\right)
\|_{L^1(\R^2)}^{1/2} .
\]
Indeed, according to Theorem \ref{WP mean field eq}, $\varphi^2\left(
f^\eps
\right)$ is locally bounded with respect to $(t,\bx)$ since $f^\eps$ has exponential moments and $\varphi$ has polynomial growth. Therefore, we obtain the result since $f^\eps$ lies in
$
\scC^0
\left( 
\R^+\times K , 
L^1
\left(\R^2\right) 
\right)$ according to Definition \ref{notion de solution}.

We now prove that $\varphi\left(
f^\eps
\right)$ has continuous time derivative. We multiply equation \eqref{kinetic:eq} by $\varphi$ and integrate with respect to $\bu$. After an integration by part, this yields
\[
\partial_t \varphi
\left(
f^\eps
\right)
 = 
\xi_1
\left(
f^\eps
\right)
 + 
\left(
\frac{1}{\eps} 
\rho_0^\eps \cV^\eps 
+ 
\Psi*_r(\rho^\eps_0 \;\cV^\eps)
\right) 
\partial_v \varphi
\left(
f^\eps
\right)
 - 
\left(
\frac{1}{\eps} 
\rho_0^\eps 
+ 
\Psi*_r\rho^\eps_0
\right) 
\xi_2
\left(
f^\eps
\right) ,
\]
with 
\[
\left\{
    \begin{array}{ll}
        \displaystyle \xi_1(\bu)
 = 
  \ds\partial_v   \varphi(\bu)
        \left(
        N(v)-w
        \right)
         + 
        \partial_w   \varphi(\bu)
        A(\bu)
         + 
        \partial_v^2   \varphi(\bu)
         
        ,
        \\[1.1em]
        \displaystyle \xi_2(\bu)
 = 
  \ds\partial_v   \varphi(\bu)
         v
         
        .
    \end{array}
\right.
\]
Functions $\ds\xi_1
\left(
f^\eps
\right),
\xi_2
\left(
f^\eps
\right),
\partial_v \varphi
\left(
f^\eps
\right),
\cV^\eps
$ are continuous according to the previous step. Furthermore, we obtain that function $\Psi*_r(\rho^\eps_0 \;\cV^\eps)$ and $\Psi*_r\rho^\eps_0$ are continuous using continuity of $\cV^\eps$ and $\rho_0^\eps$ and the assumption $\Psi\in
\scC^0
\left(
K_{\bx}, 
L^1
\left(
K_{\bx'}
\right)
\right)
$. This yields the result.
\end{proof}
We prove that when the initial data $f^\eps_0$ is smooth, the associated solution $f^\eps$ to \eqref{kinetic:eq} is regular.
\begin{proposition}\label{classical solution} Under the assumptions of Theorem \ref{WP mean field eq}, suppose in addition that $f^\eps_0$ lies in $\scC^0
\left(
K,\scC^\infty_c
\left(
\R^2
\right)
\right)
$ and that $N$ meets the following assumptions
\begin{equation}\label{hyp N reg}
    \sup_{|v| \geq 1} 
    \left| 
    N'(v)\right| | v |^{1-p} 
     <  + \infty , \qquad \sup_{|v| \geq 1} 
    \left|N'''(v)\right| 
    | v |^{-p'}
     <  + \infty .
\end{equation}
where $p$ is given in assumption \eqref{hyp:N}.
Then the solution $f^\eps$ to equation \eqref{kinetic:eq} provided by Theorem \ref{WP mean field eq} verifies 
\[
f^\eps\in
L^{\infty}_{loc}
\left(
\R^+\times K , 
W^{2,1}
\left(
\R^2
\right)
\right) ,\quad
\partial_t f^\eps\in
L^{\infty}_{loc}
\left(
\R^+\times K , 
L^{1}
\left(
\R^2
\right)
\right) .
\]
\end{proposition}
We postpone the proof to Appendix \ref{regularity estimates}: it is mainly technical and relies on moment estimates on the derivatives of $f^\eps$.\\

For self-consistency, we recall a result from \cite{BF} about the control of the macroscopic quantities 
$\ds
\left(
\cV^\eps , \cW^\eps
\right)
$ defined by \eqref{macro:q} and the error term $\ds \mathcal{E}
\left(
f^\eps
\right)
$ defined by \eqref{error}. We also provide uniform estimates with respect to $\eps$ for the moments of $f^\eps$ and  for the relative energy given by
\begin{equation*}
\left\{
\begin{array}{l}
\ds M_{q}
\left[ f^\eps 
\right](t,\bx)
 := \frac{1}{\rho_0^\eps(\bx)} \int_{\R^2}
|\bu|^{q}
 
f^\eps(t,\bx,\bu) \dD \bu ,
\\[1.1em]
\ds D_{q}
\left[ f^\eps 
\right](t,\bx) := 
\frac{1}{\rho_0^\eps(\bx)} 
\int_{\R^2} |v- \cV^\eps(t,\bx)|^{q}  
f^\eps(t,\bx,\bu) \dD \bu ,
\end{array}\right.
\end{equation*}
where $q\geq 2$.
\begin{theorem}[\cite{BF}]\label{th:preliminary}
Under assumptions \eqref{hyp:N}-\eqref{hyp2:f0} and under the additional assumptions of Theorem \ref{WP mean field eq}, consider the solutions $f^\eps$ to \eqref{kinetic:eq} provided by Theorem \ref{WP mean field eq} and the solution
$
\cU
$  to \eqref{macro:eq}. Furthermore, define the initial macroscopic error as
\[
\mathcal{E}_{\mathrm{mac}}
 = 
\left\| 
\mathcal{U}_0
 - 
\mathcal{U}_0^\eps 
\right\|_{L^{\infty}(K)}
 + 
\| 
\rho_0
-
\rho_0^\eps 
\|_{L^{\infty}(K)} .
\]
There exists 
$\ds (C , \eps_0)  \in  
\left(
\R^+_*
\right)^2
$
such that
\begin{enumerate}
    \item\label{cv macro q} for all $\eps \leq \eps_0$, it holds
\[
\left\| 
\mathcal{U}(t)
 - 
\mathcal{U}^\eps(t) 
\right\|_{L^{\infty}(K)}
 \leq 
C 
\min
{
\left( 
e^{C t}
\left( 
\mathcal{E}_{\mathrm{mac}} 
+ 
\eps 
\right) , 
1 
\right)} ,
\quad\quad
\forall t \in \R^+ ,
\]
where $\cU^\eps$ and $\cU$ are respectively given by \eqref{macro:q} and \eqref{macro:eq}.
\item\label{estimate moment mu} For all $\eps >0$ and all $q$ in 
$
\ds
[2, 2(p+p')]$ it holds
\[
M_{q}[ f^\eps ](t,\bx)
  \leq  
C , \quad\quad\forall  (t,\bx)  \in \R^+\times K,
\]
where exponent $p$ is given in assumption \eqref{hyp:N}. In particular, $\cU^\eps$, $\partial_t\,\cU^\eps$ and $\cE(f^\eps)$ are uniformly bounded with respect to both 
$\ds(t , \bx) \in \R^+\times K$ and
$\eps$, where $\ds \mathcal{E}$ is defined by \eqref{error}.
\item\label{estimate rel energy mu} For all $\eps >0$ and all $q$ in 
$
\ds
[ 2, 2(p+p') ]$ it holds
\[
D_q[ f^\eps ](t,\bx)
  \leq  
C  \left[ 
D_q[ f^\eps ](0,\bx) \exp \left(-q m_* \frac{t}{\eps}  \right)
 + 
\eps^{\frac{q}{2}} \right] , \quad\forall (t,\bx)\in\R^+\times K .
\]
\end{enumerate}
In this theorem, constants $C$ and $\eps_0$ only depend on $m_p$, $\ols{m}_p$, $m_*$ (see
\eqref{hyp:rho0}-\eqref{hyp2:f0} ) and on the data of the problem $N$, $A$ and $\Psi$.
\end{theorem}

We now complement this result with an additional estimate on the derivative of the error term, that will be used in a later proof.

\begin{proposition}\label{estimate:derivee:erreur}
For all $\eps >0$  and all $(t,\bx) \in \R^+\times K$ we have
\[
\left|
\frac{\dD}{\dD t} 
\mathcal{E}(f^\eps
\left(t,\bx,\cdot 
\right))\right| 
 \leq 
 C \left[
 \left(
 D_2
 +
 D_{p'+1}\right)
 [f^\eps](0,\bx) 
 \eps^{-1} 
 \exp \left(-2 m_* \frac{t}{\eps} \right)
 + 
1 \right] ,
\] 
which simplifies into 
\[
\left|
\frac{\dD}{\dD t} 
\mathcal{E}(f^\eps
\left(t,\bx,\cdot 
\right))\right| 
 \leq 
 C ,
\] 
as soon as assumption \eqref{hyp:f eps 0 well prepared} is fulfilled.

\end{proposition}

\begin{proof}
We compute the derivative of $\mathcal{E}(f^\eps)$ taking the difference between equation \eqref{kinetic:eq} multiplied by $N / \rho_0^\eps$ and integrated with respect to $\bu$, and the first line of \eqref{macro-eps:eq} multiplied by 
$
N'(\cV^\eps)
$. After integrating by part with respect to $v$, it yields
\begin{equation}\label{tderivative:error}
\frac{\dD}{\dD t} \mathcal{E}(f^\eps
\left(t,\bx,\cdot 
\right)) = 
\cA + \cB ,
\end{equation}
where $\cA$ and $\cB$ are given by
\begin{equation*}
\left\{
\begin{array}{l}
\ds
\cA = 
- \frac{1}{\eps} 
\int_{\R^2} 
N'(v) 
\left( v- \cV^\eps(t,\bx) 
\right) 
f^\eps(t,\bx,\bu) \dD \bu ,
\\[1.1em]
\ds \cB
 = \frac{1}{\rho_0^\eps} \int_{\R^2}
\left[ 
\left(
N' 
B^\eps
 
+
 
N''
\right) f^\eps 
\right](t,\bx,\bu) \dD \bu
 - 
N'(\cV^\eps) 
\left(
B^\eps
\left(t , \bx , \cU^\eps\right)
 + 
\cE( f^\eps )\right)
 .
\end{array}\right.
\end{equation*}

The main difficulty here consists in estimating the stiffer term $\cA$: this is what we start with. According to the definition of $\cV^\eps$, we have
\[
\cA = 
- \frac{1}{\eps} 
\int_{\R^2} 
\left(
N'(v) - 
N'
\left(\cV^\eps\right)
\right)
\left(v- \cV^\eps
\right) 
f^\eps(t,\bx,\bu) \dD \bu .
\]
Therefore, relying on our regularity assumptions on $N$, assumption \eqref{hyp N'} and item \eqref{estimate moment mu} of Theorem \ref{th:preliminary}, which ensures that $\cV^\eps$ is uniformly bounded with respect to both $(t,\bx) \in \R^+\times K$ and $\eps > 0$, we obtain some constant $C$ such that 
\[
\cA
 \leq 
\frac{C}{\eps} \left(
D_{p'+1}[ f^\eps ]
+
D_{2}[ f^\eps ]
\right) ,
\]
where we also used assumption \eqref{hyp:rho0} to bound $\rho_0^\eps$. Hence, we apply item \eqref{estimate rel energy mu} of Theorem \ref{th:preliminary} and deduce
\[
\cA
 \leq 
 C \left[ 
 \left(
 D_2
 +
 D_{p'+1}\right)
 [ f^\eps ](0,\bx) 
 \eps^{-1} 
 \exp \left(-2 m_* \frac{t}{\eps}  \right)
 + 
1 \right] ,
\]
for all $(t,\bx) \in \R^+\times K$ and $\eps > 0$. 

According to assumptions \eqref{hyp:N} and \eqref{hyp N'}, $\cB$ only displays moments of $f^\eps$ up to order $2(p+p')$, which are uniformly bounded with respect to both $(t,\bx) \in \R^+\times K$ and $\eps > 0$, according to item \eqref{estimate moment mu} of Theorem \ref{th:preliminary}. On top of that both $\cU^\eps$ and $\cE
\left( f^\eps \right)
$ are also uniformly bounded according to item \eqref{estimate moment mu} in Theorem \ref{th:preliminary}. Furthermore, according to assumptions \eqref{hyp:psi} and \eqref{hyp:rho0} on $\Psi$ and $\rho_0^\eps$, $\Psi*_r\rho^\eps_0$ is uniformly bounded as well. Therefore, there exists a constant $C$ such that
\[
\cB
  \leq  
C , 
\]
for all $(t,\bx) \in \R^+\times K$ and $\eps > 0$. 

We obtain the expected result gathering the estimates obtained on $\cA$ and $\cB$.
\end{proof}

\section{Proof of Theorem \ref{th:main}}\label{A priori estimates}

In this section, we derive uniform $L^{\infty}$ convergence estimates for the solution $\phi^\eps$ to equation \eqref{HJ eq ordre 0}. To do so, our strategy consists in performing a Hilbert expansion of $\phi^\eps$ with respect to $\eps$ and to prove that the higher order terms are uniformly bounded with respect to $\eps$. Denote by $\phi^\eps_1$ the correction of order $1$ in the expansion of $\phi^\eps$ in the regime of strong interactions
\begin{equation}\label{def phi eps 1}
\phi^\eps = -\frac{1}{2} 
\rho_0^\eps \left| v-\cV^\eps \right|^2 + 
\eps 
\phi^\eps_1 ,
\end{equation}
Plugging this \textit{ansatz} in \eqref{HJ eq ordre 0}, we find that
$\phi^\eps_1$ solves the following equation
\begin{equation}\label{HJ eq ordre 1}
 \ds\partial_t \phi^\eps_1 + \nabla_{\bu} \phi^\eps_1 \cdot \mathbf{b}^\eps + \mathrm{div}_{\bu} \left[  \mathbf{b}^\eps  \right] -
       \partial^2_v 
        \phi^\eps_1
       -
        \left| 
        \partial_v  
        \phi^\eps_1
         \right|^2
 + 
\frac{1}{\eps}\,\rho_0^\eps \left( v - \cV^\eps
\right) \partial_v  \left( \phi^\eps_1  
 -  \overline{\phi^\eps_1}\right)
 = 
0 ,
\end{equation}
where $\overline{\phi^\eps_1}$ is given by
\[
\overline{\phi^\eps_1}
\left(t,\bx,\bu\right)
 = 
n(v)
 - 
n(\cV^\eps)
 - 
\left(
v-\cV^\eps
\right)
\left(
N(\cV^\eps)
 + 
(w-\cW^\eps)
 + \cE(f^\eps)
 + 
\frac{1}{2} 
\Psi*_r \rho_0^\eps 
\left(
v-\cV^\eps
\right)
\right) ,
\]
where $n$ is the primitive of $N$ given in assumption \eqref{hyp:HJ}.
Keeping the leading order, we expect that $\phi^\eps_1$ will look like $\overline{\phi^\eps_1}$. Therefore, in the following lemma, we look for sub and super-solutions to equation \eqref{HJ eq ordre 1} with the form $\overline{\phi^\eps_1}+\psi$, where $\psi$ needs to be determined. This is done in the following lemma, which is the keystone of our analysis.
% we construct sub and super solutions for the operator 
% $\scH^\eps_1
%  + 
% \frac{1}{\eps} 
% \scJ^\eps_1 $ given by \eqref{HJ eq ordre 1}, in order to obtain $L^{\infty}$ control over the solution $\phi^\eps$ to equation \eqref{HJ eq ordre 0}.
\begin{lemma}\label{lemme technique s s solution}
Consider some positive constant $\alpha_0$ and define $\psi$ as follows
\[
\psi\left(t,\bx,\bu\right)
 = 
\frac{\alpha_0}{2} 
\left|
v-\cV^\eps(t,\bx)
\right|^2
 + 
\frac{\alpha(t)}{2} 
\left|
w-\cW^\eps(t,\bx)
\right|^2 ,
\]
where $\alpha$ is given by
\[
\alpha(t) = 
\alpha_0 
e^{
2 (|a| + b) t
}
 + 
\frac{1}{|a| + b} 
 \left( 
e^{
2 (|a| + b) t
}
 - 
1
 \right) .
\]
The following functions
\begin{equation*}
\chi_+ = \overline{\phi^\eps_1}+ \psi+ m  ,
 \qquad \chi_- = \overline{\phi^\eps_1} - \psi - m.
\end{equation*} 
are respectively \textbf{super} and \textbf{sub}-solutions to equation \eqref{HJ eq ordre 1},  where $\ds \left( t \mapsto m(t) \right)$ is given by
\[
m(t) = m_0 + C\,\exp{\left( 6 (a + b) t \right)} ,
\]
for all $m_0 \in \R$ and where the constant $C$ only depends on $\alpha_0$, the constants in  \eqref{hyp:rho0}-\eqref{hyp2:f0} and \eqref{hyp:f eps 0 well prepared}, and the data of the problem $N$, $A$ and $\Psi$.
\end{lemma}

\begin{proof} The first step of the proof consists in proving that equation \eqref{HJ eq ordre 1} evaluated in $\overline{\phi^\eps_1}$ stays uniformly bounded with respect to $\eps$. This corresponds to estimating the term $\cA$ defined as follows
\begin{equation}\label{def:cA}
\cA\,=\,\partial_t\overline{\phi^\eps_1} + \nabla_{\bu}\overline{\phi^\eps_1}\cdot\mathbf{b}^\eps+\mathrm{div}_{\bu} \left[  \mathbf{b}^\eps  \right] -\partial^2_v \overline{\phi^\eps_1}
-\left| \partial_v\overline{\phi^\eps_1}\right|^2\,.
\end{equation}
To this aim, we compute the derivatives of $\overline{\phi^\eps_1}$. We first notice that equation \eqref{macro-eps:eq} on $\cV^\eps$ can be rewritten
\begin{equation}\label{eq2:Veps}
\partial_t \cV^\eps  =  
        B^\eps(t,\bx,\cU^\eps) +\cE(f^\eps)\,,
\end{equation}
and that therefore $\overline{\phi^\eps_1}$ can be rewritten
\begin{equation*}
\overline{\phi^\eps_1}
\left(t,\bx,\bu\right)=\int_{\cV^\eps}^v \left( B^\eps\left(t,\bx,\Tilde{v},w\right)-\partial_t \cV^\eps \right)
\,\dD \Tilde{v}
\,,
\end{equation*}
where $B^\eps$ is defined in \eqref{def:Beps}. Using these relations, we find
\begin{equation*}
    \left\{
    \begin{array}{llll}
        \displaystyle &\nabla_{\bu} \overline{\phi^\eps_1} &=& \begin{pmatrix}
    B^\eps(t,\bx,\bu)-\partial_t\cV^\eps \\[0,9em]
\ds
-\left(v-\cV^\eps\right)
\end{pmatrix}\,,\\[1,6em]
         \displaystyle &\partial_{v}^2 \overline{\phi^\eps_1} &=& \mathrm{div}_{\bu} \left[  \mathbf{b}^\eps  \right]-\partial_w A(u)\,
         ,\\[0.8em]
        \ds &\partial_t\,\overline{\phi^\eps_1} &=& 
\partial_t \cV^\eps\left(\partial_t \cV^\eps-
B^\eps\left(t,\bx,\cV^\eps,w\right)\right)-(v-\cV^\eps)\partial_t^2 \cV^\eps
\,.
  \end{array}
\right.
\end{equation*}
Relying on the latter relations we rewrite
$\cA$ as follows
\begin{equation*}
 \cA\,=\,
\partial_t \cV^\eps\left(B^\eps\left(t,\bx,\bu\right)-
B^\eps\left(t,\bx,(\cV^\eps,w)\right)\right)-(v-\cV^\eps)\partial_t^2 \cV^\eps
+\partial_w A(u)
-(v-\cV^\eps)A(u)\,.
\end{equation*}
On the one hand, $B^\eps$ is given by \eqref{def:Beps}, it holds
\begin{equation*}
  B^\eps(t,\bx,\bu) - B^\eps(t,\bx,(\cV^\eps,w)) =   N(v) -  N(\cV^\eps) - \Psi*_r\rho^\eps_0  (v - \cV^\eps)\,.
\end{equation*}
On the other hand, using relation \eqref{eq2:Veps}, we obtain
\begin{equation*}
\partial_t^2 \cV^\eps  =  
        \partial_t\,\cU^\eps\cdot \nabla_{\bu} B^\eps(t,\bx,\cU^\eps) +\frac{\dD}{\dD t}\,\cE(f^\eps)\,,
\end{equation*}
which according to \eqref{def:Beps} and equation \eqref{macro-eps:eq} on $\cW^\eps$ can be rewritten
\begin{equation*}
\partial_t^2 \cV^\eps  =  
\partial_t\,\cV^\eps
\left(N'\left(\cV^\eps\right)-\Psi*_r\rho_0^\eps\right)-A\left(\cU^\eps\right)+\frac{\dD}{\dD t}\,\cE(f^\eps)\,.
\end{equation*}
Therefore, $\cA$ may be expressed as follows
\begin{equation*}
 \cA\,=\,\partial_t \cV^\eps\left( N(v) - N\left(\cV^\eps\right)\right) 
 -
\left( v - \cV^\eps \right)\left( \partial_t \cV^\eps N'\left(\cV^\eps\right) + \left( A(u) -A(\cU^\eps)\right)+\frac{\dD}{\dD t}\,\cE(f^\eps) \right)\,+\,\partial_w A\,.
\end{equation*}
On the one hand, $\partial_t \cV^\eps$ and $\cV^\eps$ are uniformly bounded according to item \eqref{estimate moment mu} in Theorem \ref{th:preliminary}. On the other hand, we apply Proposition \ref{estimate:derivee:erreur} which ensures that under the smallness assumption \eqref{hyp:f eps 0 well prepared}, the time derivative of $\cE\left(f^\eps\right)$ is uniformly bounded as well. Consequently, for all positive $\eps$, it holds
\[
\left(\left|\partial_t \cV^\eps\right|+
\left|\,\cU^\eps\right|+\left|\frac{\dD}{\dD t}\,\cE(f^\eps)\right|\right)(t,\bx)
\,\leq\,C\,,\quad\forall\,(t,\bx)\in\R^+\times K\,,
\]
for some constant $C$ depending only on the constants in assumptions 
\eqref{hyp:rho0}-\eqref{hyp2:f0}, \eqref{hyp:f eps 0 well prepared} and on the data of the problem $N$, $A$ and $\Psi$. Therefore, we deduce the following bound for $\cA$ 
\begin{equation*}
\left|\cA\right|\,\leq\,
C\left(1+\left|v-\cV^\eps\right|+\left|v-\cV^\eps\right|^2+\left| N(v)-N(\cV^\eps)\right|\right)+b\,\left|(v-\cV^\eps)(w-\cW^\eps)\right|\,.
\end{equation*}
Then we apply Young's inequality to bound the crossed term between $v$ and $w$ and use assumption \eqref{hyp:N} to bound $N(v)$. In the end, it yields
\begin{equation}\label{estimate:A}
\left|\cA\right|\,\leq\,
C\left(1+\left| v-\cV^\eps\right|^p\right)+\frac{1}{2}\,\left|w-\cW^\eps\right|^2\,.
\end{equation}
Building on this estimate, we can now pass to the heart of the proof and show that $\chi_+$ and $\chi_-$ are respectively super- and sub-solutions to \eqref{HJ eq ordre 1}. We evaluate equation \eqref{HJ eq ordre 1} in $\chi_+$ and $\chi_-$ and obtain
\begin{equation*}
\ds\partial_t\chi_{\pm}+\nabla_{\bu}\chi_{\pm}\cdot\mathbf{b}^\eps+\mathrm{div}_{\bu}\left[\mathbf{b}^\eps  \right] -\partial^2_v \chi_{\pm}-\left| \partial_v  \chi_{\pm}\right|^2+ \frac{1}{\eps}\,\rho_0^\eps \left( v - \cV^\eps\right) \partial_v \chi_{\pm}= \cA\pm\cB\pm m'(t)-\left| \partial_v  \psi\right|^2\,,
\end{equation*}
where $\cA$ is given by \eqref{def:cA} and $\cB$ is given by
\begin{equation}\label{def:cB}
\cB\,=\,\ds\partial_t \psi+\nabla_{\bu} \psi \cdot\mathbf{b}^\eps-\partial^2_v\psi-2\,\partial_v\psi\,\partial_v\overline{\phi^\eps_1}+ \frac{1}{\eps}\,\rho_0^\eps \left( v - \cV^\eps\right) \partial_v  \psi\,.
\end{equation}
In order to conclude that $\chi_+$ and $\chi_-$ are respectively sub- and super-solutions to \eqref{HJ eq ordre 1}, it is sufficient to prove
\[
\cB + m'(t)-\left| \partial_v  \psi\right|^2-\left|\cA\right|\geq0\,.
\]
Therefore, we focus on proving the latter inequality.
To begin with, we have
\begin{align*}
&\cB - \left|   \partial_v  \psi \right|^2   \\
=& -\alpha_0\left(v-\cV^\eps(\bx)\right) \partial_t\cV^\eps+\frac{\alpha'(t)}{2} \left|
w-\cW^\eps(t,\bx)\right|^2 -\alpha(t)  \left(w-\cW^\eps(t,\bx)\right)\partial_t\cW^\eps \\
&+\alpha_0 \left(v-\cV^\eps\right)B^\eps(t,\bx,\bu)+ \alpha(t)\left(w-\cW^\eps\right)A(\bu) - \alpha_0^2 \left(v-\cV^\eps\right)^2 -\alpha_0  +  \frac{\rho_0^\eps}{\eps} \alpha_0 \left(v-\cV^\eps\right)^2 \\
&- 2  \alpha_0 \left(v-\cV^\eps\right)\left( B^\eps(t,\bx,\bu)-\partial_t\cV^\eps\right)\\
=& \frac{\alpha'(t)}{2} \left|w-\cW^\eps(t,\bx)\right|^2 -\alpha_0 \left(v-\cV^\eps\right) \left( B^\eps(t,\bx,\bu) - \partial_t\,\cV^\eps\right) \\
&+ \alpha \left(w-\cW^\eps\right)\left( A(\bu) -A(\cU^\eps)\right) -\alpha_0  +  \left( \frac{\rho_0^\eps}{\eps} \alpha_0 - \alpha_0^2\right) \left(v-\cV^\eps\right)^2\\
=& \frac{\alpha'(t)}{2} \left|w-\cW^\eps(t,\bx)\right|^2 -\alpha_0 \left(v-\cV^\eps\right) \left( B^\eps(t,\bx,\bu) - B^\eps(t,\bx,\cU^\eps) -\cE(f^\eps)\right) \\
&+ \alpha \left(w-\cW^\eps\right)\left( a \left(v-\cV^\eps\right)- b\left(w-\cW^\eps\right)\right) -\alpha_0  +  \left( \frac{\rho_0^\eps}{\eps} \alpha_0 - \alpha_0^2\right) \left(v-\cV^\eps\right)^2\\
\geq & \left( \frac{\alpha'(t)}{2} -\alpha (|a|+b)\right) \left|w-\cW^\eps(t,\bx)\right|^2 -\alpha_0 \left(v-\cV^\eps\right) \left( B^\eps(t,\bx,\bu) - B^\eps(t,\bx,\cU^\eps) -\cE(f^\eps)\right) \\
&-\alpha_0  +  \left( \frac{\rho_0^\eps}{\eps} \alpha_0 - \alpha_0^2 - \frac{|a|}{4}\alpha(t) \right) \left(v-\cV^\eps\right)^2\,. 
\end{align*}
where we have used the Young inequality at the last line. 
Gathering the latter estimate and \eqref{estimate:A} we obtain
\begin{align*}
 &\ds\cB + m'(t)-\left| \partial_v  \psi\right|^2-\left|\cA\right|\\
 \geq\, &m'(t)+\left( \frac{\alpha'}{2} -\alpha (|a|+b)\right) \left|w-\cW^\eps(t,\bx)\right|^2 -\alpha_0 \left(v-\cV^\eps\right) \left( B^\eps(t,\bx,\bu) - B^\eps(t,\bx,\cU^\eps) -\cE(f^\eps)\right) \\
&-\alpha_0  +  \left( \frac{\rho_0^\eps}{\eps} \alpha_0 - \alpha_0^2 - \frac{|a|}{4}\alpha(t) \right) \left|v-\cV^\eps\right|^2-C\left(1+\left| v-\cV^\eps\right|^p\right)-\left|w-\cW^\eps\right|^2\,.\\
 \geq\, &m'(t) + \left( \frac{\alpha'}{2} -\alpha (|a|+b) - 1\right) \left|w-\cW^\eps\right|^2   -\alpha_0- \alpha_0 \left(v-\cV^\eps\right) \left( N(v) -  N(\cV^\eps)\right) \\
&+  \left(\alpha_0 \Psi*_r\rho^\eps_0(\bx) + \frac{\rho_0^\eps}{\eps} \alpha_0 - \frac{3}{2}\,\alpha_0^2 - \frac{|a|}{4}\alpha(t) -\left(|a|+\frac{b}2+ \frac{1}{2}\right) - \frac{1}{2}\,\alpha_0\right) \left(v-\cV^\eps\right)^2\\
 & + \alpha_0\left( v - \cV^\eps \right)\cE(f^\eps)  -C\left(1+\left| v-\cV^\eps\right|^p\right)\,,
 \end{align*}
where we have used Young inequality and the following relation 
\begin{equation*}
  B^\eps(t,\bx,\bu) - B^\eps(t,\bx,\cU^\eps) =   N(v) -  N(\cV^\eps) - (w -\cW^\eps) - \Psi*_r\rho^\eps_0(\bx)  (v - \cV^\eps)\,.
\end{equation*}
Now we observe crucially that according to assumption \eqref{hyp:N}, it holds
\begin{equation*}
    \left(v-\cV^\eps\right) \left( N(v) -  N(\cV^\eps)\right) \leq C - \frac{1}{C}\left\vert v-\cV^\eps\right\vert^{p+1}\,,
\end{equation*}
for some constant $C$ great enough. Furthermore, choosing $\alpha$ such that $\frac{\alpha'}{2} -\alpha (|a|+b) - 1 = 0$, that is
\[
\alpha(t) = 
\alpha_0 
\exp{
\left(
2 (|a| + b) t\right)
}
 + 
\frac{1}{ |a| + b} 
 \left( 
\exp{
\left(
2 (|a| + b) t
\right)
}
 - 
1
 \right) ,
\]
we obtain
\begin{align*}
 &\ds\cB + m'(t)-\left| \partial_v  \psi\right|^2-\left|\cA\right|\\
 \geq\, &m'(t) + \frac{1}{C}\left\vert v-\cV^\eps\right\vert^{p+1}\,
 -\,
C\,\left(1+
\exp{
\left(
2\,(|a|\,+\,b)\,t
\right)
}
\left|
v\,-\,\cV^\eps
\right|^{2}
\,+\,\left|
v\,-\,\cV^\eps
\right|^p\,
\right)\,.
 \end{align*}
Then we find that 
\begin{equation*}
\frac{1}{C}\left\vert v-\cV^\eps\right\vert^{p+1}\, -\,C\,\left(1+\exp{\left(2\,(|a|\,+\,b)\,t\right)}\left|v\,-\,\cV^\eps\right|^{2}\,+\,\left|v\,-\,\cV^\eps\right|^p\,\right)\,\geq\,-\,\Tilde{C}\,\exp{\left(\,6\,(|a|\,+\,b)\,t\,\right)}\,,
\end{equation*}
and therefore conclude the proof.
\end{proof} 
We are now able to proceed to the proof of Theorem \ref{th:main}. Indeed, relying on Lemma \ref{lemme technique s s solution} and applying a comparison principle to equation \eqref{HJ eq ordre 1}, we deduce convergence estimates for the Hopf-Cole transform $\phi^\eps$ of $f^\eps$. 
\begin{proof}[Proof of Theorem \ref{th:main}]
All along this proof, we consider some positive constants $\alpha_0$ (to be determined later on) and we work with the associated quantities $\psi$, $\chi_+$ and $\chi_-$ defined in Proposition \ref{lemme technique s s solution}. We proceed in three steps
\begin{enumerate}
    \item\label{step1} we prove that under our set of assumptions, it holds uniformly in $\eps$
    \[
    \chi_-(0,\bx,\bu)
     \leq 
    \phi_1^\eps(0,\bx,\bu)
     \leq 
    \chi_+(0,\bx,\bu) ,\quad
    \forall 
    (\bx,\bu)\in K \times \R^2 , 
    \]
    where $\chi_+$ and $\chi_-$ are defined in Proposition \ref{lemme technique s s solution}, 
    \item\label{step2}
    we apply Lemma \ref{lemme technique s s solution} and prove a comparison principle to deduce  that the latter inequality holds for all positive time, that is
    \[
    \chi_-(t,\bx,\bu)
     \leq 
    \phi_1^\eps(t,\bx,\bu)
     \leq 
    \chi_+(t,\bx,\bu) ,\quad
    \forall 
    (t,\bx,\bu)\in \R^+ \times K \times \R^2 , 
    \]
    \item\label{step3} we conclude that $\phi^\eps$ converges locally uniformly to $-\frac{1}{2}  \rho_0 \left| v-\cV \right|^2$.
\end{enumerate}

We start with step \eqref{step1}. According to assumption \eqref{hyp:HJ}, for all positive $\eps$, it holds
\begin{equation*}
n(v)
 - 
C
\left(
1+|\bu|^2
\right)
 \leq 
\left(
\phi^\eps_1
 + 
\frac{1}{\eps} 
\left( 
-\frac{1}{2} 
\rho_0^\eps \left| v-\cV^\eps \right|^2
+\frac{1}{2} 
\rho_0 \left| v-\cV \right|^2  \right) \right)(0,\bx,\bu)
 \leq 
n(v)
 + 
C
\left(
1+|\bu|^2
\right) ,
\end{equation*}
for all $(\bx,\bu)\in K\times\R^2$. On the one hand, according to assumptions \eqref{compatibility assumption}, \eqref{hyp:rho0} and \eqref{hyp1:f0}, it holds
\[
\frac{1}{\eps}
\left| 
-\frac{1}{2} 
\rho_0^\eps \left| v-\cV^\eps \right|^2
+\frac{1}{2} 
\rho_0 \left| v-\cV \right|^2 \right|(0,\bx,\bu)
 \leq 
C
\left(
1 + 
|\bu-\cU^\eps_0|^2
\right) ,
\]
for all $(\bx,\bu)\in K\times\R^2$, for some constant $C$ depending only on the initial condition $f^\eps_0$ (only through the constants appearing in assumptions \eqref{hyp:rho0}, \eqref{hyp1:f0} and \eqref{compatibility assumption}). On the other hand, 
according to assumptions \eqref{hyp:rho0} and \eqref{hyp:psi}, $\Psi*_r\rho^\eps_0$ is uniformly bounded with respect to both $\bx \in K$ and $\eps > 0$. On top of that, $\ds\cU^\eps_0$ and $\ds\cE\left(f^\eps_0\right)
$ are also uniformly bounded  with respect to both $\bx \in K$ and $\eps > 0$ according to assumptions \eqref{hyp:N} and \eqref{hyp1:f0}. Therefore,
according to the definition of $\overline{\phi^\eps_1}$ (see below \eqref{HJ eq ordre 1}) and since $N$ is continuous, it holds
\[
\left| 
n
-
\overline{\phi^\eps_1} \right|(0,\bx,\bu)
 \leq 
C
\left(
1 + 
|\bu-\cU^\eps_0|^2
\right) ,
\]
for all $(\bx,\bu)\in K\times\R^2$, for some constant $C$ depending on the initial condition $f^\eps_0$ (only through the constants appearing in assumptions \eqref{hyp:rho0}-\eqref{hyp1:f0}) and $N$. Gathering these considerations and writing
\[
\phi^\eps_1
 - 
\overline{\phi^\eps_1}
 = 
\frac{1}{\eps}
\left(
\phi^\eps
+\frac{1}{2}  \rho_0 \left| v-\cV \right|^2 - \eps n\right)
 + 
\frac{1}{\eps}
\left(
-\frac{1}{2}  \rho_0 \left| v-\cV \right|^2
+\frac{1}{2}  \rho_0^\eps \left| v-\cV^\eps \right|^2
\right)
 + n-\overline{\phi^\eps_1} ,
\]
we deduce that according to assumption \eqref{hyp:HJ}, for all positive $\eps$, it holds
\begin{equation*}
\overline{\phi^\eps_1}(0,\bx,\bu)
 - 
C
\left(
1+|\bu-\cU^\eps_0|^2
\right)
 \leq 
\phi^\eps_1(0,\bx,\bu)
 \leq 
\overline{\phi^\eps_1}(0,\bx,\bu)
 + 
C
\left(
1+|\bu-\cU^\eps_0|^2
\right) ,
\end{equation*}
for all $(\bx,\bu)\in K\times\R^2$. Therefore, taking $\alpha_0/2$ and $m(0)$ greater than $C$, we conclude step \eqref{step1}, indeed for all positive $\eps$ it holds
\[
\chi_-(0,\bx,\bu)
 \leq 
\phi_1^\eps(0,\bx,\bu)
 \leq 
\chi_+(0,\bx,\bu) ,\quad\forall (\bx,\bu)\in K\times\R^2 ,
\] 
where
$\chi_+$ and $\chi_-$ are given in Proposition \ref{lemme technique s s solution}.
~\\

Let us now turn to step \eqref{step2}, which consists in proving that the latter estimates holds true for all positive time. This step relies on two main ingredients: first, we apply Lemma \ref{lemme technique s s solution} which ensures that $\chi_-$ and $\chi_+$ are respectively sub and super-solutions to equation \eqref{HJ eq ordre 1}, then we apply a comparison argument whose proof is detailed in Appendix \ref{Appendix} and which is divided into two steps: first we apply Lemma \ref{eq ss sol HJ kin}, which ensures that the quantities $f_+$ and $f_-$ defined for all 
$
(t,\bx,\bu)\in \R^+ \times K \times \R^2
$
as
\[
f_{\pm}(t,\bx,\bu)
=
\sqrt{\frac{\rho_0(\bx)}{2\pi\eps}}
    \exp{
    \left( 
-\frac{1}{2\,\eps}\,\rho_0^\eps \left| v-\cV^\eps \right|^2
 + 
\chi_{\pm}
    \left(t, \bx, \bu\right)   
    \right)
    } ,
\]
% and
% \[
% f_{-}(t,\bx,\bu)
% =
% \sqrt{\frac{\rho_0(\bx)}{2\pi\eps}}
%     \exp{
%     \left( 
%     \left(
%     \frac{1}{\eps} 
% -\frac{1}{2}  \rho_0^\eps \left| v-\cV^\eps \right|^2
%  + 
% \chi_-\right)
%     \left(t, \bx, \bu\right)   
%     \right)
%     } ,
% \]
are respectively super and sub-solutions to equation \eqref{kinetic:eq}. 
\begin{remark}
Lemma \ref{eq ss sol HJ kin} indeed applies here since the regularity required on $f_{\pm}$ is ensured by Lemma \ref{WP2}.
\end{remark}
We come back to the proof, and notice that according to the previous step, it holds
\[
f_{-}(0,\bx,\bu)
 \leq 
f^\eps_0(\bx,\bu)
 \leq 
f_{+}(0,\bx,\bu) ,\quad\forall (\bx,\bu)\in K\times\R^2 .
\]      
Therefore, according to Lemma \ref{comp principle kin eq}, we obtain
\[
f_{-}(t,\bx,\bu)
 \leq 
f_0^\eps(t,\bx,\bu)
 \leq 
f_{+}(t,\bx,\bu) ,\quad\forall (t,\bx,\bu)\in \R^+\times K\times\R^2 .
\]      
We deduce that the bound obtained in step \eqref{step1}, propagates through time, that is, for all positive $\eps$, it holds
\[
\chi_-(t,\bx,\bu)
 \leq 
\phi_1^\eps(t,\bx,\bu)
 \leq 
\chi_+(t,\bx,\bu) ,\quad\forall (t,\bx,\bu)\in \R^+\times K\times\R^2 .
\]      
~\\
We can now turn to the last step and prove our main result. According to the definition of $\phi^\eps_1$ and the result of step \eqref{step2}, it holds
\[
-\frac{1}{2}  \rho_0^\eps \left| v-\cV^\eps \right|^2
+\frac{1}{2}  \rho_0 \left| v-\cV \right|^2
 + 
\eps 
\chi_-
 \leq 
\phi^\eps
+\frac{1}{2}  \rho_0 \left| v-\cV \right|^2 
 \leq 
-\frac{1}{2}  \rho_0^\eps \left| v-\cV^\eps \right|^2
+\frac{1}{2}  \rho_0 \left| v-\cV \right|^2
 + 
\eps 
\chi_+ ,
\] 
for all
$\ds
(t,\bx,\bu)\in \R^+\times K\times\R^2 
$. On the one hand, relying on item \eqref{cv macro q} in Theorem \ref{th:preliminary} and since the initial condition $f^\eps_0$ meets the compatibility assumption \eqref{compatibility assumption}, there exists two positive constants $C$ and $\eps_0$ such that for all $\eps \leq \eps_0$, it holds
\[
\left|
-\frac{1}{2}  \rho_0^\eps \left| v-\cV^\eps \right|^2
+\frac{1}{2}  \rho_0 \left| v-\cV \right|^2\right|(t,\bx,\bu)
 \leq 
\eps Ce^{Ct}
\left(
1 + 
|\bu|^2
\right) ,
\]
for all $(t,\bx,\bu)\in \R^+\times K\times\R^2$, where constants $C$ and $\eps_0$ only depends on the data of our problem: $f^\eps_0$ (only through the constants appearing in assumptions \eqref{hyp:rho0}-\eqref{hyp2:f0} and \eqref{compatibility assumption}), $N$, $A$ and $\Psi$. On the other hand, 
according to assumption \eqref{hyp:rho0} and \eqref{hyp:psi}, $\Psi*_r\rho^\eps_0$ is uniformly bounded with respect to both $\bx \in K$ and $\eps > 0$. On top of that, $\ds\cU^\eps$ and $\ds\cE\left(f^\eps\right)
$ are also uniformly bounded  with respect to both $(t,\bx) \in \R^+\times K$ and $\eps > 0$ according item \eqref{estimate moment mu} in Theorem \ref{th:preliminary}. Therefore,
according to the definitions of $\chi_-$ and $\chi_+$ (see Lemma \ref{lemme technique s s solution}) and since $N$ is continuous, it holds
\[
\left| 
\chi_{\pm}
-
n \right|(t,\bx,\bu)
 \leq 
Ce^{Ct}
\left(
1 + 
|\bu|^2
\right) ,
\]
for all $(t,\bx,\bu)\in K\times\R^2$, where the constant $C$ only depends on the data of our problem: $f^\eps_0$ (only through the constants appearing in assumptions \eqref{hyp:rho0}-\eqref{hyp2:f0}), $N$, $A$ and $\Psi$. Hence, we deduce the result: for all $\eps \leq \eps_0$ it holds
\[
\eps
\left( 
n(v)
 - 
Ce^{Ct}
\left(
1 + 
|\bu|^2
\right)
\right)
 \leq 
\left(\phi^\eps
+\frac{1}{2}  \rho_0 \left| v-\cV \right|^2 \right)(t,\bx,\bu)
 \leq 
\eps
\left( 
n(v)
 + 
Ce^{Ct}
\left(
1 + 
|\bu|^2
\right)
\right) ,
\] 
for all
$\ds
(t,\bx,\bu)\in \R^+\times K\times\R^2 
$, where constants $C$ and $\eps_0$ only depends on the data of our problem: $f^\eps_0$ (only through the constants appearing in assumptions \eqref{hyp:rho0}-\eqref{hyp2:f0} and \eqref{compatibility assumption}-\eqref{hyp:f eps 0 well prepared}), $N$, $A$ and $\Psi$.
\end{proof}

\section*{Acknowledgment}
The authors thank warmly Francis Filbet all the discussions without which it would not have been possible to achieve this work, Philippe Laurençot for his enlightening suggestions and comments and Sepideh Mirrahimi for her precious explanations regarding Hamilton-Jacobi equations. A.B. gratefully acknowledge the support of  ANITI (Artificial and Natural Intelligence Toulouse Institute). This project has received support from ANR ChaMaNe No: ANR-19-CE40-0024.

\appendix
\addcontentsline{toc}{section}{Appendices}

\section{Comparison principles}\label{Appendix}
The object of this section is to prove a comparison principle for equation \eqref{HJ eq ordre 1} in order to complete step \eqref{step2} in the proof of Theorem \ref{th:main}. More precisely, we prove that if the quantity $\phi^\eps_1$ defined by \eqref{def phi eps 1} verifies 
\[
    \chi_-(0,\bx,\bu)
     \leq 
    \phi_1^\eps(0,\bx,\bu)
     \leq 
    \chi_+(0,\bx,\bu) ,\quad
    \forall 
    (\bx,\bu)\in K \times \R^2 , 
    \]
where $\chi_-$ an $\chi_+$ are respectively sub and super-solutions to \eqref{HJ eq ordre 1}, then the latter estimate propagates through time, that is
\[
    \chi_-(t,\bx,\bu)
     \leq 
    \phi_1^\eps(t,\bx,\bu)
     \leq 
    \chi_+(t,\bx,\bu) ,\quad
    \forall 
    (t,\bx,\bu)\in \R^+ \times K \times \R^2 .
    \]
Instead of working directly on equation \eqref{HJ eq ordre 1}, our strategy consists in proving a comparison principle for the following linearized version of the kinetic equation \eqref{kinetic:eq}
\begin{equation}
  \label{linearkinetic:eq}
  \ds\partial_t   f
         + 
        \mathrm{div}_{\bu}
        \left[  
        \mathbf{b}^\eps
         f 
         \right]
        - 
       \partial^2_v  
        f
         = 
        \frac{1}{\eps} 
        \rho^\eps_0 
        \partial_v 
        \left[ 
        (v-\cV^\eps)
         f 
        \right] 
        .
\end{equation}
Indeed, it is more convenient to work on equation \eqref{linearkinetic:eq} since we can rely on the decaying properties of solutions to \eqref{kinetic:eq} provided by Theorem \ref{WP mean field eq}. From the comparison principle on equation \eqref{linearkinetic:eq}, we will easily deduce the expected result. This approach is made possible since, according to the following lemma, there is a direct link between sub and super-solutions to equations \eqref{linearkinetic:eq} and \eqref{HJ eq ordre 1}
\begin{lemma}\label{eq ss sol HJ kin}
Consider some fixed $\eps > 0$. Under the assumptions of Theorem \ref{WP mean field eq}, consider the solution $f^\eps$ to equation \eqref{kinetic:eq} and its associated macroscopic quantities 
$
\left(
\cV^\eps , \cW^\eps
\right)
$
provided by Theorem \ref{WP mean field eq}. Furthermore, consider a strictly positive function $f$ such that 
$
f
$,
$
\partial_t f
$ and
$
\nabla^2_{\bu} f
$ lie in $\scC^0
\left(
\R^+\times
K\times
\R^2
\right)
$ and define $\chi$ as follows
\[
f(t,\bx,\bu)
=
\sqrt{\frac{\rho_0(\bx)}{2\pi\eps}}
    \exp{\left( \left(
-\frac{1}{2\eps}  \rho_0^\eps \left| v-\cV^\eps \right|^2
 + 
\chi\right)
    \left(t, \bx, \bu\right)   
    \right)
    } ,\quad
    \forall 
    (t,\bx,\bu)\in \R^+ \times K \times \R^2 , 
\]
Then the following statements are equivalent
\begin{enumerate}
    \item $f$ is a super-solution (resp. sub-solution) to \eqref{linearkinetic:eq}.
    \item $\chi$ is a super-solution (resp. sub-solution) to \eqref{HJ eq ordre 1}.
\end{enumerate}
More precisely it holds
\begin{equation*}
  \ds\partial_t   f
         + 
        \mathrm{div}_{\bu}
        \left[  
        \mathbf{b}^\eps
         f 
         \right]
        - 
       \partial^2_v  
        f
         - 
        \frac{1}{\eps} 
        \rho^\eps_0 
        \partial_v 
        \left[ 
        (v-\cV^\eps)
         f 
        \right]
         = 
        \left(
        \scH^\eps_1 + 
\frac{1}{\eps} 
\scJ^\eps_1\right)\left[
 \phi_1
\right] f ,
\end{equation*}
for all $(t,\bx,\bu)\in \R^+ \times K \times \R^2$.
\end{lemma}
\begin{proof}
We consider $f$ and $\chi$ as in Lemma \ref{eq ss sol HJ kin} and proceed in two steps. On the one hand, plugging $f$ in equation \eqref{linearkinetic:eq}, one has the following relation
\begin{equation*}
  \ds\partial_t   f
         + 
        \mathrm{div}_{\bu}
        \left[  
        \mathbf{b}^\eps
         f 
         \right]
        - 
       \partial^2_v  
        f
         - 
        \frac{1}{\eps} 
        \rho^\eps_0 
        \partial_v 
        \left[ 
        (v-\cV^\eps)
         f 
        \right]
         = \frac{1}{\eps} f\,\cA
        \,,
\end{equation*}
for all $(t,\bx,\bu)\in \R^+ \times K \times \R^2$ where $\cA$ gathers the terms obtained plugging $\ds \phi\,:=\,-\frac{1}{2}  \rho_0^\eps \left| v-\cV^\eps \right|^2
 + \eps 
\chi$ into equation \eqref{HJ eq ordre 0}, that is
\begin{equation*}
\cA\,=\,
\partial_t \phi + \nabla_{\bu}\phi \cdot
\mathbf{b}^\eps + \eps  \mathrm{div}_{\bu} \left[  \mathbf{b} \right] - \partial^2_v \phi - \rho_0^\eps -  \frac{1}{\eps} \left(  \partial_v \left(
\frac{1}{2} 
\rho_0^\eps \left| v-\cV^\eps \right|^2 + \phi  \right) \partial_v \phi \right)\,.
\end{equation*}
On the other hand, according to computations already detailed at the beginning of Section \ref{A priori estimates}, $\cA$ also corresponds to the terms obtained plugging $\chi$ into equation \eqref{HJ eq ordre 1}, that is 
\begin{equation*}
        \cA
\,=\, 
\eps\,
\left(
\partial_t \chi + \nabla_{\bu} \chi \cdot \mathbf{b}^\eps + \mathrm{div}_{\bu} \left[  \mathbf{b}^\eps  \right] -
       \partial^2_v 
        \chi
       -
        \left| 
        \partial_v \chi
         \right|^2
 + 
\frac{1}{\eps}\,\rho_0^\eps \left( v - \cV^\eps
\right) \partial_v  \left( \chi 
 -  \overline{\phi^\eps_1}\right)
\right)\,,
\end{equation*}
for all $(t,\bx,\bu)\in \R^+ \times K \times \R^2$. Gathering these two relations, we obtain the result.
\end{proof}

It is now left to prove that a comparison principle holds for equation \eqref{linearkinetic:eq}. It is the object of the following Lemma 
\begin{lemma}\label{comp principle kin eq}
Consider some fixed $\eps > 0$. Under the assumptions of Theorem \ref{WP mean field eq}, consider the solution $f^\eps$ to equation \eqref{kinetic:eq} and its associated macroscopic quantities 
$
\left(
\cV^\eps , \cW^\eps
\right)
$
provided by Theorem \ref{WP mean field eq}. Furthermore, consider a strictly positive function $f$ such that 
$
f
$,
$
\partial_t f
$ and
$
\nabla^2_{\bu} f
$ lie in $\scC^0
\left(
\R^+\times
K\times
\R^2
\right)
$. Suppose that at initial time, it holds
\[
    f^\eps_0(\bx,\bu)
     \leq 
    f(0,\bx,\bu)
     ,\quad
    \forall 
    (\bx,\bu)\in K \times \R^2 , 
    \]
and that $f$ is super-solution to equation \eqref{linearkinetic:eq}, that is
\begin{equation*}
0 \leq 
  \ds\partial_t   f
         + 
        \mathrm{div}_{\bu}
        \left[  
        \mathbf{b}^\eps
         f 
         \right]
        - 
       \partial^2_v  
        f
         - 
        \frac{1}{\eps} 
        \rho^\eps_0 
        \partial_v 
        \left[ 
        (v-\cV^\eps)
         f 
        \right]
\end{equation*}
for all $(t,\bx,\bu)\in \R^+ \times K \times \R^2$. Then it holds
\[
    f^\eps(t,\bx,\bu)
     \leq 
    f(t,\bx,\bu)
     ,\quad
    \forall 
    (t,\bx)\in \R^+\times K ,\text{a.e. in }\bu\in\R^2 .
    \]
Furthermore, the latter statement is also true if we replace "super-solution" by "sub-solution" and the symbol "$\geq$" by "$\leq$".
\end{lemma}
\begin{proof}
Since the proof of the comparison principle for super and sub-solutions is the same, we only detail it for super-solutions.
Let us first suppose that $f^\eps_0$ lies in $\scC^0
\left(
K,\scC^\infty_c
\left(
\R^2
\right)
\right)
$ so that, according to Proposition \ref{classical solution}, $f^\eps$ is a classical solution to \eqref{kinetic:eq}. Then it holds
\begin{equation*}
  \ds\partial_t   (f^\eps-f)
         + 
        \mathrm{div}_{\bu}
        \left[  
        \mathbf{b}^\eps
         (f^\eps-f) 
         \right]
        - 
       \partial^2_v  
        (f^\eps-f)
         - 
        \frac{1}{\eps} 
        \rho^\eps_0 
        \partial_v 
        \left[ 
        (v-\cV^\eps)
         (f^\eps-f) 
        \right]
         \leq 0
\end{equation*}
$a.e.$ with respect to $(t,\bx,\bu)\in \R^+ \times K \times \R^2$. Therefore, multiplying the latter relation by $\ds \mathds{1}_{f^\eps\geq f}$, we deduce that in the weak sense, it holds
\begin{equation*}
  \ds\partial_t   (f^\eps-f)_+ + \mathrm{div}_{\bu} \left[   \mathbf{b}^\eps (f^\eps-f)_+  \right]
        - 
       \partial^2_v  
        (f^\eps-f)_+
         - 
        \frac{1}{\eps} 
        \rho^\eps_0 
        \partial_v 
        \left[ 
        (v-\cV^\eps)
         (f^\eps-f)_+ 
        \right]
         \leq 0
\end{equation*}
where $(\cdot)_+$ stands for the positive part and is defined by $(\cdot)_+ = (|\cdot|+id_{\R})/2$. In order to derive the latter relation, we follow a classical procedure which we do not detail here and which consists in regularizing the positive part $(\cdot)_+$. Then we integrate the latter relation with respect to time and $\bu$ and obtain 
\begin{equation*}
\int_{\R^2}(f^\eps-f)_+(t,\bx,\bu) \dD \bu \leq 
\int_{\R^2}(f^\eps-f)_+(0,\bx,\bu) \dD \bu .
\end{equation*}
The latter computations are justified since
$
0 \leq (f^\eps-f)_+ \leq f^\eps
$ and since, according to Theorem \ref{WP mean field eq}, $f^\eps$ has moments up to any order with respect to $\bu$. Since $f^\eps \leq f$ at time $t=0$, we deduce the result.\\

To conclude the proof, we lift the condition $f^\eps_0 \in
\scC^0\left(K,
\scC^\infty_c\left(\R^2\right)
\right)$ using the $\scC^0
\left(
K,
L^1\left(\R^2\right)
\right)
$-continuity of equation \eqref{kinetic:eq} with respect to the initial data (see \cite[Section A.2]{BF} for a detailed proof of this continuity result) and following a density argument.
\end{proof}

\section{Regularity estimates}\label{regularity estimates}
In this section, we derive regularity estimates for the solution $f^\eps$ to equation \eqref{kinetic:eq} and therefore prove Proposition \ref{classical solution}. The main difficulty here consists in dealing with the contribution due to the non-linear drift $N$. We manage to bypass this difficulty by estimating the norm of $f^\eps$ and its derivatives in the following weighted $L^1$ spaces
\[
L^{1}(\omega_q)= \left\{ f:\R^2\mapsto\R , \int_{\R^2} |f| \omega_{q}(\bu) \dD\bu < +\infty\right\},
\]
where $\ds\omega_q(\bu) = 1 + |\bu|^q$, for $q \geq 2$.
The first step consists in estimating the norm of $f^\eps$ in $L^1(\omega_q)$. This step relies on previous results obtained in \cite{BF}. Then we adapt these computations to evaluate the norm of the derivatives of $f^\eps$ in $L^1(\omega_q)$. Indeed, the derivatives of $f^\eps$ solve equation \eqref{kinetic:eq} with additional source terms whose contribution can be controlled thanks to the confining properties of the non-linear drift $N$. Let us outline the strategy in the case of the first order derivatives: equation \eqref{kinetic:eq} on $f^\eps$ reads as follows
\[
\partial_t f^\eps = \scA^\eps f^\eps ,
\]
where operator $\scA^\eps$ is given by
\[
\scA^\eps f = \partial^2_v  f+ \frac{1}{\eps} \rho^\eps_0 \partial_v \left[ (v-\cV^\eps)f \right]- \mathrm{div}_{\bu}\left[   \mathbf{b}^\eps f \right] .
\]
Relying the arguments developed in the proof of \cite[Proposition 3.1]{BF} it holds
\begin{lemma}\label{estime op A}
Consider some fixed $\eps$ and some $q \geq 2$. Under the assumptions of Theorem \ref{WP mean field eq}, consider the operator $\scA^\eps$ associated to the solution $f^\eps$ to equation \eqref{kinetic:eq}. There exists a positive constant $C$ such that for all $\bx \in K$, it holds
\[
\int_{\R^2} \sign{(f)} 
\scA^\eps\left(f\right) \omega_{q}(\bu) \dD\bu
 \leq 
C \|f\|_{L^1
\left(
\omega_q
\right)
} + 
q\int_{\R^2}
\mathds{1}_{|v| \geq 1}
\frac{N(v)}{v} |v|^q f(\bu) \dD \bu
 ,
\]
for all function $f$ lying in $W^{2,1}(\omega_{p+q})$.
\end{lemma}
As a direct consequence of Lemma \ref{estime op A}, we deduce that for any smooth initial data $f^\eps_0$, the solution $f^\eps$ to equation \eqref{kinetic:eq} provided by Theorem \ref{WP mean field eq} lies in 
$
L^\infty_{loc}
\left(
\R^+\times K , L^1
\left(
\omega_q
\right)
\right)
$ for all exponent $q \geq 2$. Indeed, since $N$ meets the confining assumption \eqref{hyp:N} we have:
$
\ds
\mathds{1}_{|v| \geq 1} N(v)/v
 \leq C
$ for some constant $C$. Therefore, multiplying equation \eqref{kinetic:eq} by $\sign{
\left(
f^\eps
\right)
} \omega_q(\bu)$, integrating with respect to $\bu$ and applying Lemma \ref{estime op A}, we obtain that for all exponent $q$, it holds
\[
\frac{\dD}{\dD t} 
\|f^\eps\|_{L^1
\left(
\omega_q
\right)
}
 \leq C \|f^\eps\|_{L^1
\left(
\omega_q
\right)
} .
\]
Hence, applying Gronwall's lemma and taking the supremum over all $\bx \in K$, we deduce
\begin{equation}\label{estime Omega f}
\|f^\eps(t)\|_{L^\infty
\left(K , 
L^1
\left(
\omega_q
\right)\right)
}
 \leq e^{Ct} \|f^\eps_0\|_{L^\infty
\left(K , 
L^1
\left(
\omega_q
\right)\right)
} ,
\end{equation}
for all $t\in \R^+$. We follow the same strategy for derivatives of $f^\eps$:
writing $\ds (g,h) = 
\left(\partial_v f^\eps,\partial_w f^\eps\right)$ and differentiating equation \eqref{kinetic:eq} with respect to $v$ and $w$ we obtain 
\begin{equation}\label{eq:grad f eps}
\left\{
    \begin{array}{ll}
        \ds\partial_t   g
         = 
        \scA^\eps g
         - 
        \left(
        N'-\Psi*_r\rho_0^\eps-\frac{1}{\eps} \rho_0^\eps
        \right) g
         - N'' f^\eps
         - 
        a h
         
        ,
        \\[1.1em]
        \ds\partial_t   h
         = 
        \scA^\eps h
         + g
         + b h
         .
    \end{array}
\right.
\end{equation}
Therefore, $g$ and $h$ solve the same equation as $f^\eps$ with additional an high order term due to the non-linear drift $N$. We control this additional term thanks to the confining properties of $N$.
\begin{proof}[Proof of Proposition \ref{classical solution}]
Let us consider some initial data $f^\eps_0$ lying in $\scC^0
\left(
K , 
\scC^\infty_c
\left(
\R^2
\right)
\right)
$ and the associated solution $f^\eps$ to \eqref{kinetic:eq} provided by Theorem \ref{WP mean field eq}. We start by proving that 
$
f^\eps
$ lies in 
$\ds
L^{\infty}_{loc}
\left(
\R^+\times K , 
W^{1,1}
\left(
\R^2
\right)
\right)
$. We fix some $\bx \in K$, some exponent $q \geq 2$ and integrate with respect to $\bu$ the sum between the first equation in \eqref{eq:grad f eps} multiplied by $\sign{(g)} \omega_q(\bu)$ and the second multiplied by $\sign{(h)} \omega_q(\bu)$. According to Lemma \ref{estime op A}, assumptions \eqref{hyp:psi} on $\Psi$, \eqref{hyp:rho0} on $\rho_0^\eps$ and \eqref{hyp N'} on $N$, we obtain
\begin{align*}
\frac{\dD}{\dD t} 
\left(
\|g\|_{L^1
\left(
\omega_q
\right)
}+
\|h\|_{L^1
\left(
\omega_q
\right)
}\right)
 \leq &C
\left(
\|g\|_{L^1
\left(
\omega_q
\right)
}
+
\|h\|_{L^1
\left(
\omega_q
\right)
}
+
\|f^\eps\|_{L^1
\left(
\omega_{p'}
\right)
}\right)\\[0.8em]
&+
\int_{\R^2}
\left(
q \mathds{1}_{|v| \geq 1}
\frac{N(v)}{v}
-N'(v)
\right)|v|^q g(\bu) \dD \bu
 ,
\end{align*}
where $p'$ is given in assumption \eqref{hyp N'}. Since $N$ meets assumptions \eqref{hyp:N} and \eqref{hyp N reg}, we deduce that for $q$ great enough, it holds
\begin{align*}
\frac{\dD}{\dD t} 
\left(
\|g\|_{L^1
\left(
\omega_q
\right)
}+
\|h\|_{L^1
\left(
\omega_q
\right)
}\right)
 \leq C
\left(
\|g\|_{L^1
\left(
\omega_q
\right)
}
+
\|h\|_{L^1
\left(
\omega_q
\right)
}
+
\|f^\eps\|_{L^1
\left(
\omega_{p'}
\right)
}\right) .
\end{align*}
Therefore, applying Gronwall's lemma, taking the supremum over all $\bx \in K$ and replacing $\ds\|f^\eps\|_{L^1
\left(
\omega_{p'}
\right)
}$ according to estimate \eqref{estime Omega f}, we deduce
\begin{align*}
\|g(t)\|_{L^\infty
\left(K , 
L^1
\left(
\omega_q
\right)\right)
}+
\|h(t)\|&_{L^\infty
\left(K , 
L^1
\left(
\omega_q
\right)\right)
}
 \leq\\[0,8em]
&e^{Ct} 
\left(
\|g_0\|_{L^\infty
\left(K , 
L^1
\left(
\omega_q
\right)\right)
}
+
\|h_0\|_{L^\infty
\left(K , 
L^1
\left(
\omega_q
\right)\right)
}
+
\|f^\eps_0\|_{L^\infty
\left(K , 
L^1
\left(
\omega_q
\right)\right)
}\right) ,
\end{align*}
for all time $t\in\R^+$. As a straightforward consequence, we obtain the expected result:
$\ds
f^\eps\in
L^{\infty}_{loc}
\left(
\R^+\times K , 
W^{1,1}
\left(
\R^2
\right)
\right)
$.\\~\\

We obtain that $\ds
f^\eps\in
L^{\infty}_{loc}
\left(
\R^+\times K , 
W^{2,1}
\left(
\R^2
\right)
\right)
$ iterating the same procedure as before but this time on the derivatives of $g$ and $h$ and using assumption \eqref{hyp N reg}, which ensures that $N'''$ has polynomial growth.\\

To end with, we obtain $\ds
\partial_t f^\eps\in
L^{\infty}_{loc}
\left(
\R^+\times K , 
L^{1}
\left(
\R^2
\right)
\right)
$ noticing that according to the definition of $\cA^\eps$, it holds
\[
\|\scA^\eps f^\eps\|_{L^{\infty}
\left(K , 
L^1
\left(
\R^2
\right)\right)
}
 \leq 
C 
\|f^\eps\|_{L^{\infty}
\left(K , 
W^{2,1}
\left(
\omega_q
\right)\right)
} ,
\]
for $q$ great enough. Then we apply the previous estimates to the relation
\[
\partial_t f^\eps = \scA^\eps f^\eps .
\]
\end{proof}

\end{document}